\documentclass[12pt]{article}
\usepackage{epsfig}
\usepackage{amsfonts}
\usepackage{amssymb}
\usepackage{amsmath}
\usepackage{amsthm}
\usepackage{latexsym}
\usepackage{graphicx}
\usepackage{bm}
\usepackage{tabularx}
\usepackage{booktabs} 
\usepackage{enumerate}
\usepackage{caption}
\usepackage[titletoc]{appendix}
\usepackage[numbers]{natbib}
\usepackage{bbm}
\usepackage{url}
\usepackage{subfig}
\usepackage{hyperref}
\usepackage{color}
\usepackage{xcolor}
\usepackage[section]{placeins}
\usepackage[thinlines]{easytable}
\usepackage{afterpage}

\catcode`\@=11

\theoremstyle{plain}
\newtheorem{theorem}{Theorem}[section]

\theoremstyle{definition}

\theoremstyle{remark}

\graphicspath{ {C:/Users/doldop/Dropbox/Philip_Doldo/Mean Field/DDE System Plots/} }

\begin{document}

\title{Queues with Delayed Information: Analyzing the Impact of the Choice Model Function}
\author{   
Philip Doldo \\ Center for Applied Mathematics \\ Cornell University
\\ 657 Rhodes Hall, Ithaca, NY 14853 \\  pmd93@cornell.edu  \\ 
\and
  Jamol Pender \\ School of Operations Research and Information Engineering \\ Center for Applied Mathematics \\ Cornell University
\\ 228 Rhodes Hall, Ithaca, NY 14853 \\  jjp274@cornell.edu  \\ 
 }


\maketitle
\begin{abstract}

In this paper, we study queueing systems with delayed information that use a generalization of the multinomial logit choice model as its arrival process. Previous literature assumes that the functional form of the multinomial logit model is exponential. However, in this work we generalize this to different functional forms. In particular, we compute the critical delay and analyze how it depends on the choice of the functional form. We highlight how the functional form of the model can be interpreted as an exponential model where the exponential rate parameter is uncertain. Furthermore, the rate parameter distribution is given by the inverse Laplace-Stieltjes transform of the functional form when it exists. We perform numerous numerical experiments to confirm our theoretical insights.
\end{abstract}


\section{Introduction} \label{CHOICE_MODEL_sec_intro}

The field of queueing theory has been dedicated to understanding the dynamics of queueing systems and their impact on society. In an age where information is readily available, it is an important problem to understand how information affects the decision making of customers who want to join queueing systems. In some applications, such as amusement parks, transportation systems, and even telecommunication systems, it has been observed that the information given to the customers can potentially be delayed \citet{doldo2021queues, novitzky2020update, pender2020stochastic, pender2018analysis, nirenberg2018impact}. 
   This delay in information can be caused by factors such as the fact that it can take time to process and send information about queue lengths to the customers and thus the information that customers actually receive describes the queue lengths from some amount of time in the past. Queueing systems where customers are provided with delayed information have been studied extensively, see for example \citet{novitzky2020limiting, doldo2021mean, doldo2020multi, doldo2021breaking, novitzky2020queues, novitzky2019nonlinear, Lipshutz2018, whitt2021many}. 

There are many different types of delayed information that customers can be provided with and different choices can lead to different choice models and thus different queueing system dynamics. One of the most common choices used in the literature is where the customer is provided with the lengths of the queues from some amount of time in the past. This delay in information is often modeled with a constant delay, as is the case in \citet{mitzenmacher2000useful, raina2005buffer, kamath2015car, raina2015stability, lipshutz2015existence, lipshutz2017exit, Lipshutz2018, nirenberg2018impact, novitzky2020limiting, doldo2021breaking, novitzky2019nonlinear, pender2017queues, pender2018analysis, whitt2021many}. 

Different types of queue length information with a constant delay have been used, see \citet{doldo2021mean} for example, which considers so-called mean-field queue length information with a constant delay. Other types of information that have been considered include delayed velocity information and moving average information, see \citet{dong2018impact, novitzky2020limiting, pender2017queues, novitzky2019nonlinear}. Some of the literature has examined queueing systems where the delay is not constant. For example, a queueing model that uses delayed updating information is examined in \citet{doldo2021queues, novitzky2020update}. Additionally, queueing models with distributed delays have been studied by considering distributed delay equation models in \citet{doldo2020multi, novitzky2020queues}, though it is worth noting that distributed delay models exhibit different dynamics in general than simply replacing the constant delay with a random variable in standard models, as discussed in \citet{doldo2021note}. 

In this paper, we will assume a similar model in spirit of \citet{pender2017queues}. However, our model uses a generalized version of the multinomial logit choice model which typically assumes an exponential form. The exponential form is related to a linear utility function, however our functional form will correspond to a nonlinear utility function. One of our main goals of this work is to understand how this nonlinearity of the utility function affects the critical delay which induces a Hopf bifurcation. As discussed in more detail in Section \ref{CHOICE_MODEL_hopf_section}, we will focus our attention on utilities that incorporate the complementary cumulative distribution function of some probability distribution. We will examine how the choice of probability distribution used to induce the choice model can lead to different system dynamics. 

 \subsection{Main Contributions of Paper}

The contributions of this work can be summarized as follows:    
\begin{itemize}
\item We consider queueing system choice models informed by more general utilities that are built from the complementary cumulative distribution functions of some probability distribution.
\item We solve for an equilibrium solution and compute the critical delay at which the queueing system exhibits a change in stability in terms of the hazard function corresponding to the given probability distribution.
\item We numerically examine how the queueing dynamics change when different probability distributions induce the choice model, including how stability properties differ between various common probability distributions.
\end{itemize} 


\subsection{Organization of Paper}

The remainder of this paper is organized as follows. In Section \ref{CHOICE_MODEL_model_introduction_section} we introduce the class of queueing models that we will be analyzing. In Section \ref{CHOICE_MODEL_hopf_section}, we find an equilibrium solution and compute the critical delay of our queueing system. In Section \ref{CHOICE_MODEL_eigenvalue_section} we consider various probability distributions and how their corresponding choice models impact the dynamics of the queueing system. In Section \ref{CHOICE_MODEL_conclusion_section} we give concluding remarks and discuss possible extensions.


\section{Generalized Multinomial Logit Queueing Model}
\label{CHOICE_MODEL_model_introduction_section}

In this section we describe the DDE system that we are interested in analyzing. The previous literature examines delay differential equation (DDE) fluid models that are based on the following model of an $N$-queue system 
\begin{eqnarray}
\overset{\bullet}{q}_i(t) &=& \lambda \frac{\exp(- \theta q_i(t - \Delta))}{\sum_{j=1}^{N} \exp( - \theta q_j(t - \Delta))} - \mu q_i(t), \hspace{5mm} i = 1, ..., N
\label{CHOICE_MODEL_system1}
\end{eqnarray} where $q_i(t)$ represents the length of the $i^{\text{th}}$ queue at time $t$ for $i = 1, ..., N$, $\lambda > 0$ is the arrival rate, and $\mu > 0$ is the service rate of the infinite-server queue, $\theta > 0$ is a choice model parameter, and $\Delta > 0$ is a time delay. The arrival rate $\lambda$ is multiplied by the probability of joining the $i^{\text{th}}$ queue which is determined by the choice model which, in this case, is based on a multinomial logit model. The choice model takes the delayed queue lengths as input to model providing customers with delayed information about the queue lengths. More specifically, the logit choice model used in Equation \ref{CHOICE_MODEL_system1} assumes that the utility of an agent joining the $i^{\text{th}}$ queue is 
\begin{eqnarray}
U_i = - \theta q_i(t - \Delta) + \epsilon_i
\end{eqnarray}
where $\epsilon_i$ is distributed according to a standard Gumbel distribution. See \cite{train2009discrete} for the details regarding how to construct such a choice model. We note that the deterministic portion of the utility is negative so that maximizing the utility corresponds to joining a queue with a shorter length (neglecting the random component of the utility as some realizations of the utility could be larger despite having a larger queue length if the realization of the corresponding random component of the utility happens to be small enough). 

In this paper, we will focus on analyzing the more general DDE system 
\begin{eqnarray}
\overset{\bullet}{q}_i(t) &=& \lambda \frac{ \bar{G}(q_i(t - \Delta))}{\sum_{j=1}^{N} \bar{G} ( q_j(t - \Delta))} - \mu q_i(t), \hspace{5mm} i = 1, ..., N
\label{CHOICE_MODEL_dde_system}
\end{eqnarray}
which uses a more general functional form of the multinomial logit model which is based on a more general utility in the form
\begin{eqnarray}
U_i = \log \left( \bar{G} ( q_i(t - \Delta))  \right) + \epsilon_i
\end{eqnarray}
where $\bar{G} : \mathbb{R} \to [0, \infty)$ is a decreasing function (so that longer queue lengths correspond to smaller utilities, neglecting randomness) and $\epsilon_i$ is again distributed according to a standard Gumbel distribution. This allows the utility to depend nonlinearly on the delayed queue length. We note that the choice model used in Equation \ref{CHOICE_MODEL_system1} corresponds to when \begin{eqnarray}
\bar{G}(x) = \exp(-\theta x)
\end{eqnarray} which happens to be the complementary cumulative distribution function corresponding to an exponential distribution with parameter $\theta$. This motivated us to focus on the case where $\bar{G}$ is a complementary cumulative distribution function corresponding to some probability distribution as this may be a particularly interesting class of decreasing functions to consider. Under such a choice model, the probability of joining the $i^{\text{th}}$ queue is \begin{eqnarray}
\frac {\bar{G} (q_i(t - \Delta))}{\sum_{j=1}^{N} \bar{G} (q_j(t - \Delta))}.
\end{eqnarray}

Additionally, some complementary cumulative distribution functions can be written as the moment-generating function of some random variable $Y < 0$ (negativity of $Y$ is important so that the moment-generating function is a decreasing function) so that 
\begin{eqnarray}
\bar{G}(x) = M_Y(x) := \mathbb{E}[e^{ Y x}]
\end{eqnarray}
\noindent where the expectation is taken over the random variable $Y$. An example of such a case is when the complementary cumulative distribution function corresponds to a hyperexponential distribution so that
\begin{eqnarray}
\bar{G}(x) = \sum_{k=1}^{m} p_k e^{- \theta_k x} = \mathbb{E}[e^{Y x}]
\end{eqnarray} where $Y$ is a discrete random variable where $Y = - \theta_k < 0$ with probability $p_k$ for $k=1,...,m$. In such a case, we note that the deterministic part of the utility can be viewed as the cumulant-generating function of the random variable $Y$, that is 
\begin{eqnarray}
U_i = \log(\mathbb{E}[e^{Y q_i(t - \Delta)}]), \hspace{5mm} i=1,..., N.
\end{eqnarray}
We can interpret such a utility as one corresponding to an exponential distribution where we have some uncertainty about the parameter (modeled by the random variable $-Y > 0$) of the distribution and account for this uncertainty by averaging over the distribution that we assume the parameter follows. If we view the utility as a nonlinear transformation of a linear quantity (namely, $ - Y q_i(t-\Delta)$), then we can view this as averaging over the possible values for the slope of this linear quantity. Of course, such choices of $\bar{G}$ can be interpreted as Laplace transforms. For ease of notation, let $Z := -Y > 0$ so that
\begin{align}
\bar{G}(x) = M_Y(x) &= \mathbb{E}[e^{Yx}]\\
&= \mathbb{E}[e^{-Zx}]\\
&= \int_{0}^{\infty} e^{-zx} f_Z(z) dz\\ 
&= \mathcal{L}[f_Z](x)
\end{align}
where $f_Z$ is the probability density function of $Z$ and $\mathcal{L}$ denotes the Laplace transform. It follows that we can directly recover the density $f_Z$ from $\bar{G}(x)$ by applying the inverse Laplace transform so that 
\begin{eqnarray}
f_Z = \mathcal{L}^{-1}[\bar{G}].
\label{CHOICE_MODEL_inverse_laplace_eqn}
\end{eqnarray}
This relation gives us a way of interpreting the choice models corresponding to various distributions as being induced by an exponential distribution with uncertainty in its parameter. Ultimately, we will concern ourselves with better understanding the dynamics of queueing systems with choice models induced by different probability distributions.


\section{Queueing System Stability Analysis}
\label{CHOICE_MODEL_hopf_section}

In this section, we will perform a stability analysis of the queueing system with the generalized functional form of the multinomial logit model. We will find an equilibrium solution of the system \ref{CHOICE_MODEL_dde_system} and use this equilibrium solution to perform a linearization which will allow us to compute the critical delay of this system.

It is well-known that when treating the delay $\Delta$ as a bifurcation parameter, the DDE system \ref{CHOICE_MODEL_system1} undergoes a Hopf bifurcation. In particular, there exists a critical delay value $\Delta_{\text{cr}}$ such that if $\Delta > \Delta_{\text{cr}}$, then the system is unstable an oscillates periodically. Computing this critical delay is an important part of understanding the dynamics of the queueing system. In this work, we will focus our attention on the choice model and consider how different choice models can impact the system dynamics. In particular, we will think of a choice model as being induced by a probability distribution in the sense that the probability of joining the $i^{\text{th}}$ queue can be written in the form \begin{eqnarray}
\frac {\bar{G}(q_i(t - \Delta))}{\sum_{j=1}^{N} \bar{G}(q_j(t - \Delta))}
\end{eqnarray}
where $\bar{G}$ is the complementary cumulative distribution function corresponding to the given probability distribution. As we noted above, the choice model in system \ref{CHOICE_MODEL_system1} is induced by an exponential distribution with parameter $\theta$. We will examine how the dynamics of the queueing system change when the choice model is induced by different probability distributions.

Suppose we are given some random variable $X : \Omega \to \mathbb{R}$ on the probability space $(\Omega, \mathcal{F}, \mathbb{P})$ and we assume that $X$ has a cumulative distribution function $G(x) = \mathbb{P}(X \leq x)$ and a corresponding probability density function $g$ exists. We define the complementary cumulative distribution function to be $\bar{G}(x) = 1 - G(x) = \mathbb{P}(X > x)$. We also define the so-called hazard function corresponding to this distribution to be $h(x) = \frac{g(x)}{\bar{G}(x)}$. We consider the DDE system 
\begin{eqnarray}
\overset{\bullet}{q}(t) &=& \lambda \frac{\bar{G}(q_i(t - \Delta))}{\sum_{j=1}^{N} \bar{G}(q_j(t - \Delta))} - \mu q_i(t), \hspace{5mm} i = 1, ..., N.
\end{eqnarray}

\noindent We are interested in better understanding how using complementary cumulative distribution functions of different probability distributions to build the choice model can impact the dynamics of the queueing system. We note that the system \ref{CHOICE_MODEL_dde_system} is a generalization of the system \ref{CHOICE_MODEL_system1}, the latter of which had its choice model induced by an exponential distribution with parameter $\theta > 0$. Understanding how the critical delay of the system changes as different choice models are used is very important to understanding how the system dynamics change as the critical delay tells us when the system is stable or unstable. In order to find the critical delay of the system \ref{CHOICE_MODEL_dde_system}, we will perform a linearization analysis for which an equilibrium solution is needed. Thus, we introduce an equilibrium solution with Theorem \ref{CHOICE_MODEL_equilibrium_theorem}. 

\begin{theorem}
The system \ref{CHOICE_MODEL_dde_system} has an equilibrium solution at $q_1 = \cdots = q_N = q^* $ where $$q^* = \frac{\lambda}{N \mu}.$$ 
\label{CHOICE_MODEL_equilibrium_theorem}
\end{theorem}
\begin{proof}
Letting $q_1 = \cdots = q_N = q^*$ for some constant $q^* \in \mathbb{R}$ to be determined, we get the system of $N$ identical equations in the form
\begin{align}
0 &= \lambda \frac{\bar{G}(q^*)}{\sum_{j=1}^{N} \bar{G}(q^*)} - \mu q^*\\
&= \lambda \frac{\bar{G}(q^*)}{N \bar{G}(q^*)} - \mu q^*\\
&= \frac{\lambda}{N} - \mu q^*
\end{align}
and so it follows that $q^* = \frac{\lambda}{N \mu}$ is a solution.
\end{proof}

A nice property of system \ref{CHOICE_MODEL_dde_system} is that the equilibrium solution from Theorem \ref{CHOICE_MODEL_equilibrium_theorem} always exists regardless of the probability distribution that is used to define the choice model. Now that we have an equilibrium solution, we can perform the linearization analysis used to get Theorem \ref{CHOICE_MODEL_delta_cr_theorem}. Before proceeding, we want to note that the hazard function $h(x)$ corresponding to the given probability distribution is an important quantity that will show up in our analysis and thus we make a brief comment on hazard functions. We can rewrite the hazard function corresponding to the distribution of a continuous random variable $X$ as follows
\begin{align}
    h(x) = \frac{g(x)}{\bar{G}(x)} = \frac{g(x)}{\mathbb{P}(X > x)} &= \lim_{\delta \to 0} \frac{G(x + \delta) - G(x)}{\delta} \cdot \frac{1}{\mathbb{P}(X > x)}\\
    &= \lim_{\delta \to 0} \frac{\mathbb{P}(x < X \leq x + \delta)}{\delta} \cdot \frac{1}{ \mathbb{P}(X > x)}\\
    &= \lim_{\delta \to 0} \frac{1}{\delta} \frac{\mathbb{P}( \{ X \leq x + \delta \} \cap \{ X > x  \} )}{\mathbb{P}(X > x)}\\
    &= \lim_{\delta \to 0} \frac{\mathbb{P}(X \leq x + \delta | X > x)}{\delta}.
\end{align}

\begin{theorem}
The system \ref{CHOICE_MODEL_dde_system} exhibits a change in stability at the critical value $\Delta_{\text{cr}}$ of the delay parameter $\Delta$ where 
\begin{eqnarray}
\Delta_{\text{cr}} &=& \frac{\text{arccos} \left( \frac{\mu}{C} \right)}{\sqrt{C^2 - \mu^2}} 
\end{eqnarray}
where
\begin{eqnarray}
C = -  \frac{\lambda}{N} \frac{g \left( \frac{\lambda}{N \mu}  \right)}{\bar{G} \left( \frac{\lambda}{N \mu}  \right)} = - \frac{\lambda}{N} h \left( \frac{\lambda}{N \mu}  \right)
\end{eqnarray}
and $\Delta_{\text{cr}}$ is valid when $C < - |\mu|$.

\label{CHOICE_MODEL_delta_cr_theorem}
\end{theorem}

\begin{proof}

We begin by linearizing the system \ref{CHOICE_MODEL_dde_system} about the equilibrium point $q_1 = \cdots = q_N = q^*$ where $q^* = \frac{\lambda}{N \mu}$ by introducing the variables $$q_i(t) = q^* + u_i(t), \hspace{5mm} i = 1, ..., N.$$ Letting $u(t) = (u_1(t), ..., u_N(t))^T$, we obtain the linearized system
\begin{eqnarray}
\overset{\bullet}{u}(t) &=& A u(t - \Delta) - \mu u(t)
\end{eqnarray}
where $$A = C \begin{bmatrix}
\frac{N-1}{N} & - \frac{1}{N}  & - \frac{1}{N}  \dots  - \frac{1}{N}\\
- \frac{1}{N} & \frac{N-1}{N} & - \frac{1}{N} \dots  - \frac{1}{N}\\
- \frac{1}{N} & - \frac{1}{N} & \frac{N-1}{N} \dots - \frac{1}{N}\\
\vdots & & \ddots   \vdots\\
- \frac{1}{N} & - \frac{1}{N} & - \frac{1}{N} \dots \frac{N-1}{N}
\end{bmatrix}$$
where $$C = -  \frac{\lambda}{N} \frac{g \left( \frac{\lambda}{N \mu}  \right)}{\bar{G} \left( \frac{\lambda}{N \mu}  \right)} = - \frac{\lambda}{N} h \left( \frac{\lambda}{N \mu}  \right).$$ The matrix $A$ has eigenvalues $0$ with multiplicity $1$ and $C$ with multiplicity $N - 1$ and we can diagonalize the system according to the change of variables $u(t) = Ev(t)$ where $v(t) = (v_1(t), ..., v_N(t))^T$ and $E$ is a matrix whose columns are eigenvectors of $A$ so that $AE = DE$ where $D = \text{diag}(0, C, ..., C)$ is a diagonal matrix with the eigenvalues of $A$ on its diagonal. Applying this change of variables yields
\begin{align}
\overset{\bullet}{u}(t) &= A u(t - \Delta) - \mu u(t)\\
&\iff \nonumber \\
E \overset{\bullet}{v}(t) &= A E v(t - \Delta) - \mu Ev(t)\\
&\iff \nonumber \\
E\overset{\bullet}{v}(t) &= E D v(t - \Delta) - \mu Ev(t)\\
&\iff \nonumber \\
\overset{\bullet}{v}(t) &= D v(t - \Delta) - \mu v(t)
\end{align}
which can be written as 
\begin{align}
\overset{\bullet}{v}_1(t) &= -\mu v_1(t) \label{CHOICE_MODEL_stable_eqn}\\
\overset{\bullet}{v}_2(t) &= C v_2(t - \Delta) - \mu v_2(t)\\
&\vdots\\
\overset{\bullet}{v}_N(t) &= C v_N(t - \Delta) - \mu v_N(t)
\end{align}

The first of these equations, Equation \ref{CHOICE_MODEL_stable_eqn}, is an ordinary differential equation with a solution that decays with time since $\mu > 0$. We thus turn our attention to the remaining $N - 1$ equations with are all DDEs in the same form. By assuming a solution in the form $v_j(t) = e^{rt}$ for $j \in {2, ..., N}$ and $r \in \mathbb{C}$, we get the following characteristic equation.
\begin{eqnarray}
r - Ce^{-r \Delta} + \mu &=& 0 \label{CHOICE_MODEL_characteristic_eqn}
\end{eqnarray}
Stability changes when the real part of $r$ changes signs and thus we consider when $r$ is on the imaginary axis so that $r = i \omega$ for some $\omega \in \mathbb{R}$. Separating into real and imaginary parts yields
\begin{align}
\mu &= C \cos(\omega \Delta)\\
\omega &= - C \sin(\omega \Delta)
\end{align}
so that $$\mu^2 + \omega^2 = C^2$$ and we take $$\omega = \sqrt{C^2 - \mu^2}.$$
Looking back at the characteristic equation \ref{CHOICE_MODEL_characteristic_eqn}, we have that 
\begin{align}
r - Ce^{-r \Delta} + \mu &= 0\\
&\iff \nonumber \\
\frac{C}{\mu + i \omega} &= e^{i \omega \Delta}\\
&\iff \nonumber\\
\frac{C (\mu - i \omega)}{\mu^2 + \omega^2} &= e^{i \omega \Delta}\\
&\iff \nonumber \\
\frac{\mu}{C} - i \frac{\omega}{C} &= e^{i \omega \Delta}\\
&\iff \nonumber \\
\log \left( \frac{\mu}{C} - i \frac{\omega}{C} \right) &= i \omega \Delta\\
&\iff \nonumber \\
\ln \left(\frac{\mu^2 + \omega^2}{C^2} \right) + i \text{arg} \left(  \frac{\mu}{C} - i \frac{\omega}{C}  \right) &= i \omega \Delta\\
&\iff \nonumber \\
\Delta &= \frac{1}{\omega} \text{arg} \left(  \frac{\mu}{C} - i \frac{\omega}{C}  \right)
\end{align}
where we note that $ \frac{\mu}{C} - i \frac{\omega}{C}$ is of unit modulus and thus we can alternatively write
\begin{eqnarray}
\Delta_{\text{cr}} = \frac{1}{\omega} \text{arccos} \left( \frac{\mu}{C} \right)
\end{eqnarray} 
provided that $\text{arccos} \left( \frac{\mu}{C} \right) \in \mathbb{R}$ because we require that $\text{arg} \left(  \frac{\mu}{C} - i \frac{\omega}{C}  \right) \in \mathbb{R}$ by definition. This requirement forces $\left| \frac{\mu}{C} \right| \leq 1$ or equivalently that $C^2 \geq \mu^2$. Additionally, we require that $\omega \in \mathbb{R}$ so that this inequality becomes strict so that $C^2 > \mu^2$. One can show that the system is always unstable when $C > |\mu|$ and we require that $C < -|\mu|$. These conditions ensure that a real value of $\Delta_{\text{cr}}$ exists. To account for multivaluedness, we take the smallest value of $\Delta_{\text{cr}} > 0$.

For completeness, we can write the critical delay as 
\begin{eqnarray}
\Delta_{\text{cr}} = \frac{\text{arccos} \left( \frac{\mu}{C} \right)}{\sqrt{C^2 - \mu^2}} = \frac{ \text{arccos} \left( - \frac{N \mu \bar{G} \left( \frac{\lambda}{N \mu}  \right) }{ \lambda g \left( \frac{\lambda}{N \mu}  \right) }  \right) }{ \sqrt{  \left( \frac{ \lambda g \left( \frac{\lambda}{N \mu}  \right) }{N  \bar{G} \left( \frac{\lambda}{N \mu}  \right) }   \right)^2 - \mu^2} }.
\end{eqnarray}

\end{proof}

As seen in the proof of Theorem \ref{CHOICE_MODEL_delta_cr_theorem}, we can express the critical delay in terms of the repeated eigenvalue $C$ of the linearized system. In particular, we see that $C$ directly depends on the hazard function of the given probability distribution. Using this information, we can explore how using different probability distributions to create the choice model can give the system different stability properties. 


\section{Dynamics for Different Distributions}
\label{CHOICE_MODEL_eigenvalue_section}

In this section, we numerically examine the dynamics of the queueing system \ref{CHOICE_MODEL_dde_system} for complementary cumulative distribution functions $\bar{G}$ corresponding to various different probability distributions. In analyzing the dynamics of the queueing system, it is crucial to understand the critical delay as this provides important information about the qualitative behavior of the system. As seen in Section \ref{CHOICE_MODEL_hopf_section}, the critical delay of the queueing system depends directly on the eigenvalue $C$ of the linearized system according to the relation 
\begin{eqnarray}
\Delta_{\text{cr}} = \frac{\text{arccos} \left( \frac{\mu}{C} \right)}{\sqrt{C^2 - \mu^2}}
\end{eqnarray}
where the eigenvalue $C$ is directly related to the hazard rate of the chosen probability distribution by 
\begin{eqnarray}
C = - \frac{\lambda}{N} h \left( \frac{\lambda}{N \mu} \right).
\end{eqnarray}
Therefore, this eigenvalue is a useful quantity to examine to help understand the qualitative system dynamics. In the remainder of this section, we consider when the given probability distribution is exponential, normal, log-normal, Weibull, gamma, or of phase-type. For each of these distributions, we will provide useful quantities including the probability density function, denoted $g$, the complementary cumulative distribution function, denoted $\bar{G}$, the mean, the variance, and the eigenvalue $C$. Additionally, we will numerically examine how the critical delay depends on the mean and variance of the given distribution.

\subsection{Exponential Distribution}

In this section, we assume that the complementary cumulative distribution function $\bar{G}$ that characterizes the choice model is given by an exponential distribution. Below we provide some useful quantities relating to the exponential distribution.

\begin{align}
X &\sim \text{Exp}(\theta), \hspace{5mm} \theta > 0\\
g(x) &= \theta e^{- \theta x}, \hspace{5mm} x \geq 0\\
\bar{G}(x) &= e^{- \theta x}\\
\mathbb{E}[X] &= \frac{1}{\theta}\\
\text{Var}(X) &= \frac{1}{\theta^2}\\
C &= - \frac{\lambda \theta}{N}
\end{align}

The exponential distribution is a well-known nonnegative continuous probability distribution. The DDE system \ref{CHOICE_MODEL_dde_system} corresponding to an exponential distribution is equivalent to a DDE queueing system with a multinomial logit choice model, which has already been extensively studied in the existing literature. Such a choice model uses $-\theta q_i(t-\Delta)$ for the deterministic part of the utility of joining the $i^{\text{th}}$ queue so that the utility depends linearly on the delayed queue length. We note that the exponential distribution is determined by a single parameter $\theta > 0$ as we see that its mean and variance are both fully determined by this parameter. Additionally, the exponential distribution has a constant hazard function given by $h(x) = g(x)/\bar{G}(x) = \theta$.


The exponential complementary cumulative distribution function can be viewed as the moment generating function of a random variable $Y$ that takes on the value $- \theta$ with probability $1$ and thus, in accordance with Equation \ref{CHOICE_MODEL_inverse_laplace_eqn}, we have
\begin{eqnarray}
f_{-Y}(x) = \mathcal{L}^{-1}[e^{-\theta x}] = \delta(x - \theta)
\end{eqnarray}
where $\delta(x)$ denotes a Dirac delta centered at the origin. Obviously, this is a trivial case where we have complete certainty of the parameter of our exponential distribution.

In Figure \ref{CHOICE_MODEL_fig_exponential} we show queue length plots and phase diagrams on each side of the critical delay. In particular, we observe that the queue lengths approach an equilibrium when the delay is below the critical delay and the queue lengths oscillate with some limiting amplitude when the delay is larger than the critical delay. It is worth noting that the mean and variance of the exponential distribution are determined by a single parameter and thus the exponential distribution has a fixed variance for a fixed mean. In Figure \ref{CHOICE_MODEL_fig_exponential_mean} we see how the value of the critical delay varies as the mean of the distribution is varied. We see that there is a critical value of the mean that, when approached, can make the critical delay arbitrarily large. This is useful information because instability can lead to inefficiencies in the queueing system as the resulting queue length oscillations can cause some servers to be overworked and others to be underworked. Thus, being able to make the critical delay large would be beneficial for the overall productivity of the queueing system as large delays could still result in a stable system.

\begin{figure}
\begin{tabular}{cc}
  \includegraphics[scale=.55]{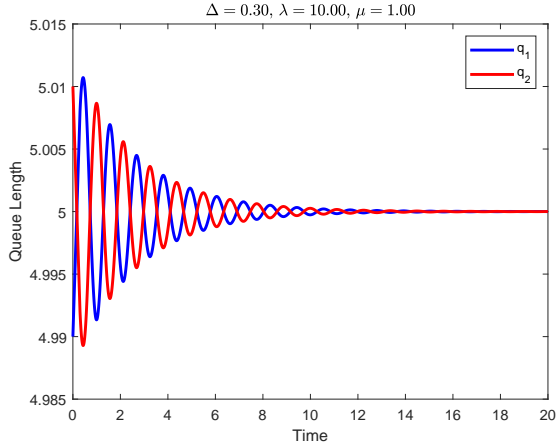} &   \includegraphics[scale=.55]{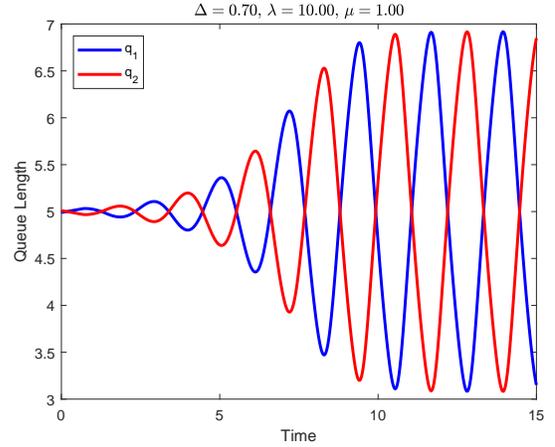} \\
(a)  & (b) \\[6pt]
 \includegraphics[scale=.55]{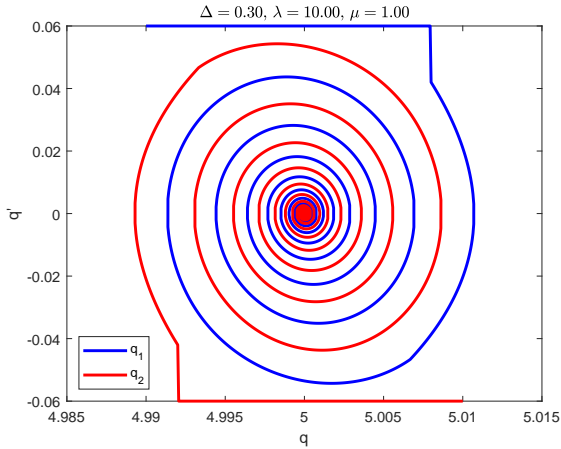} &   \includegraphics[scale=.55]{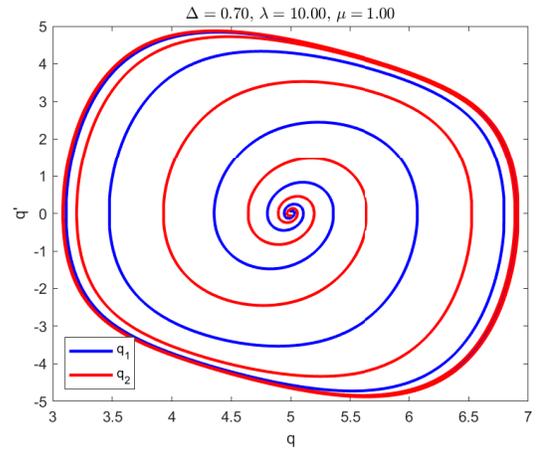} \\
(c)  & (d)  \\[6pt]
\end{tabular}
\caption{Before and after the change in stability using the choice model induced by an \textbf{exponential distribution} with $\theta = 1$ with constant history function on $[-\Delta, 0]$ with $q_1 = 4.99$ and $q_2 = 5.01$, $N = 2, \lambda = 10$, $\mu = 1$. The top two plots are queue length versus time with $\Delta = .3$ (a) and $\Delta = .7$ (b). The bottom two plots are phase plots of the queue length derivative with respect to time against queue length for $\Delta = .3$ (c) and $\Delta = .7$ (d). The critical delay is $\Delta_{\text{cr}} = .3617$.}
\label{CHOICE_MODEL_fig_exponential}
\end{figure}



\begin{figure}
\hspace{-10mm} \includegraphics[scale=.6]{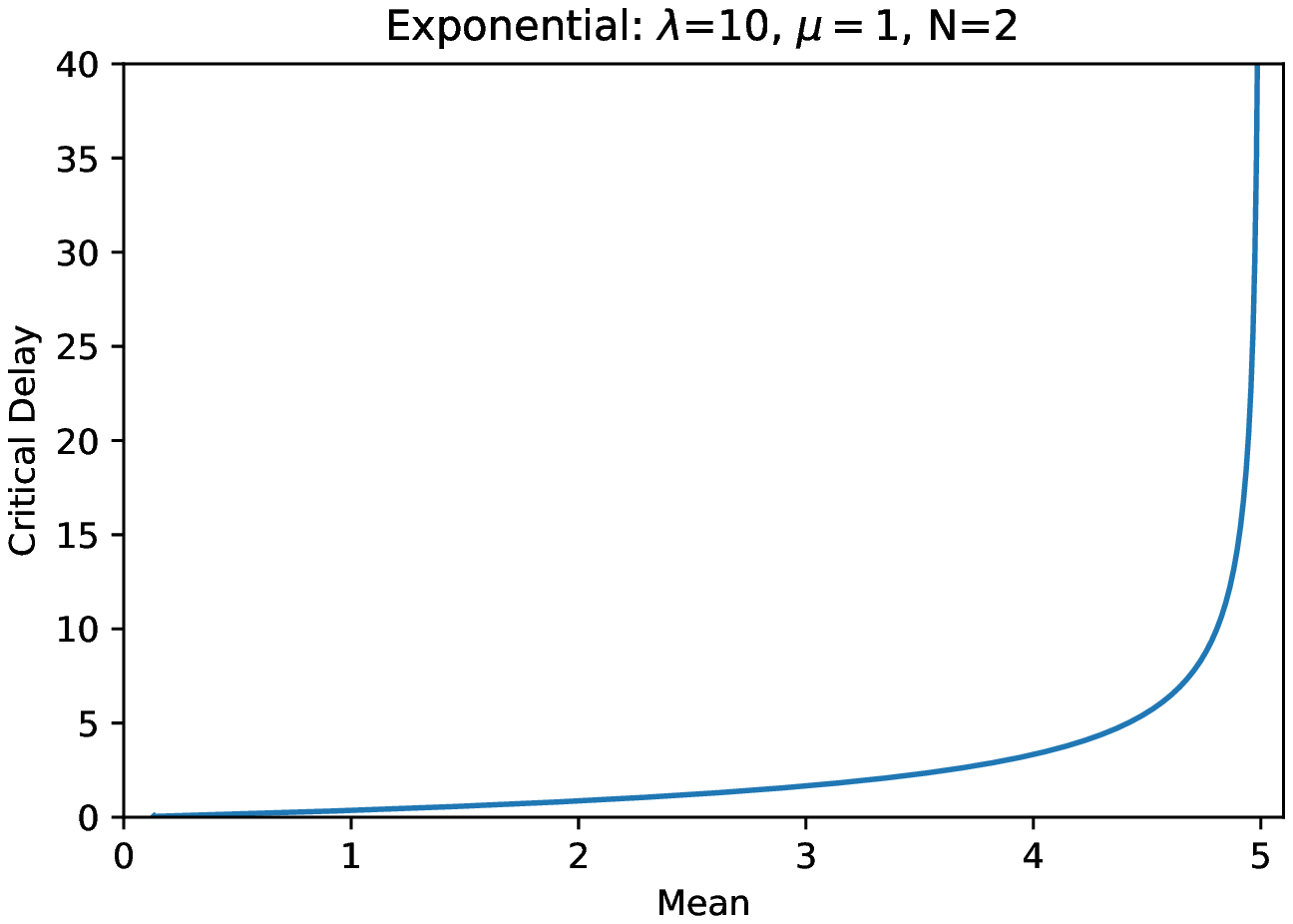} \includegraphics[scale=.6]{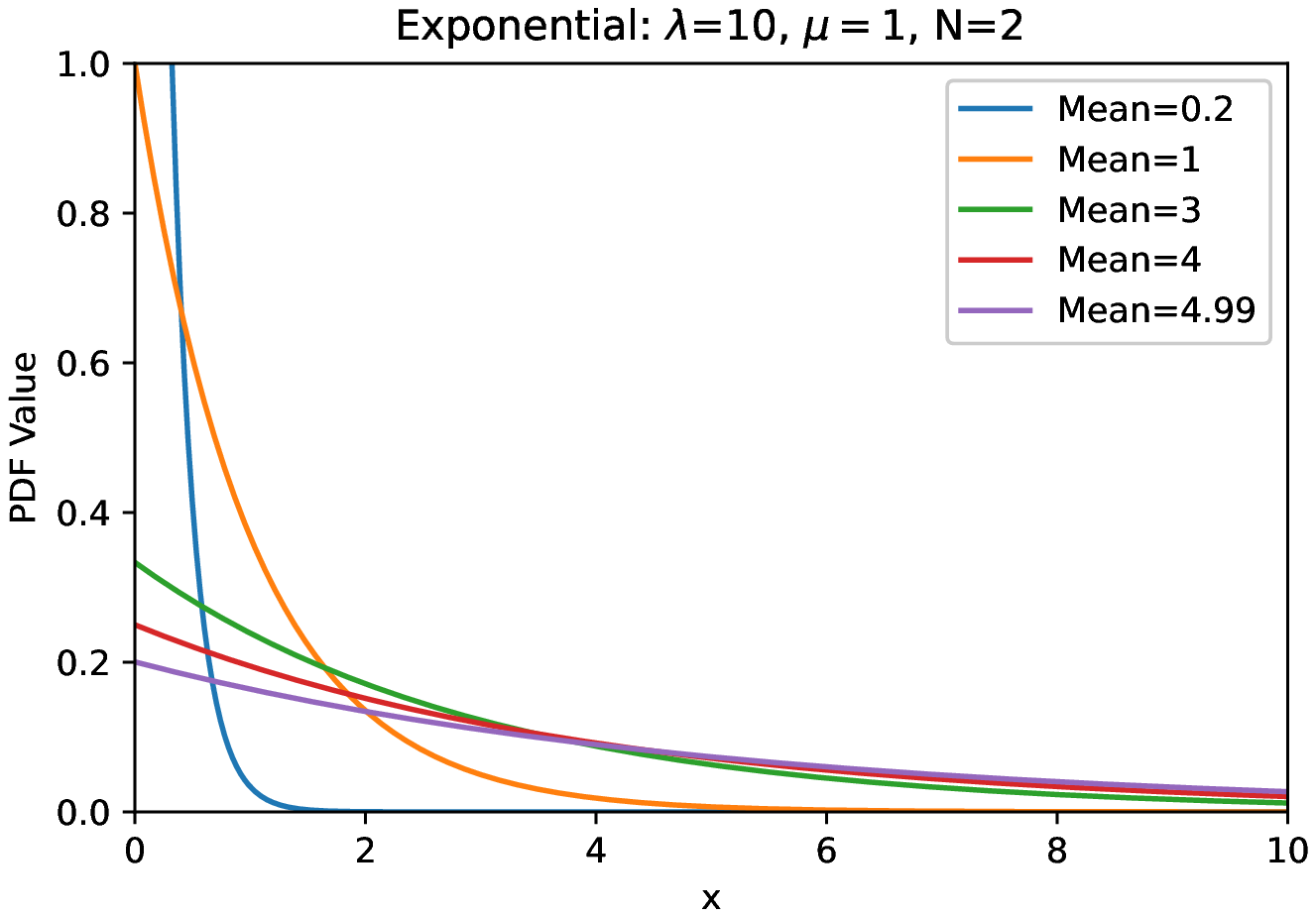}
\caption{Left: The critical delay plotted against the mean of the \textbf{exponential distribution} that induces the choice model. Right: A plot of the probability density function used for some selected values of the mean.}
\label{CHOICE_MODEL_fig_exponential_mean}
\end{figure}


\subsection{Normal Distribution}

In this section, we assume that the complementary cumulative distribution function $\bar{G}$ that characterizes the choice model is given by a normal distribution. Below we provide some useful quantities relating to the normal distribution.

\begin{align}
X &\sim \text{Normal}(\alpha, \sigma^2)\\
g(x) &= \frac{1}{\sigma \sqrt{2 \pi}} \exp \left( - \frac{1}{2} \left( \frac{x - \alpha}{\sigma}  \right)^2  \right)\\
\bar{G}(x) &= \frac{1}{2} \left[ 1 - \text{erf} \left( \frac{x - \alpha}{\sigma \sqrt{2}} \right)  \right]\\
\mathbb{E}[X] &= \alpha\\
\text{Var}(X) &= \sigma^2\\
C &= - \sqrt{\frac{2}{\pi}}\frac{\lambda}{N \sigma} \frac{\exp \left( - \frac{1}{2} \left( \frac{ \frac{\lambda}{N \mu} - \alpha}{\sigma}  \right)^2  \right)}{\left[ 1 - \text{erf} \left( \frac{ \frac{\lambda}{N \mu} - \alpha}{\sigma \sqrt{2}} \right)  \right] }
\end{align}

An important difference between the normal distribution and the exponential distribution is that the normal distribution is characterized by two parameters instead of just one parameter, its mean parameter $\alpha$ and standard deviation parameter $\sigma$. This is interesting to consider because we can vary the variance of the normal distribution for a fixed mean and vice versa and thus we have the freedom to adjust the value of the critical delay even if one of these parameters remains fixed.

The normal distribution has an interesting connection to hazard functions. In particular, if one considers the mean of a so-called truncated normal distribution, one can show that for $X \sim \text{Normal}(\alpha, \sigma^2)$
\begin{align}
    \mathbb{E}[X | X < a] &= \alpha - \sigma h \left( \frac{a - \alpha}{\sigma} \right)\\
    \mathbb{E}[X | X > a] &= \alpha + \sigma h \left( \frac{a - \alpha}{\sigma} \right)
\end{align}
\noindent where $h(x)$ is the hazard function of the standard normal distribution. This is true largely because the probability density function $\phi$ of the standard normal distribution has the nice property that $\frac{d}{dx} \phi(x) = -x \phi(x)$.

In Figure \ref{CHOICE_MODEL_fig_normal} we show queue length plots and phase diagrams on each side of the critical delay. In Figure \ref{CHOICE_MODEL_fig_normal_mean} we see how the value of the critical delay varies as the mean is varied with a fixed variance and in Figure \ref{CHOICE_MODEL_fig_normal_variance} we see how it changes as the variance is varied with a fixed mean. We see that in the cases considered, there is a critical value at which the critical delay can be made arbitrarily large when approached. 

\begin{figure}
\begin{tabular}{cc}
  \includegraphics[scale=.55]{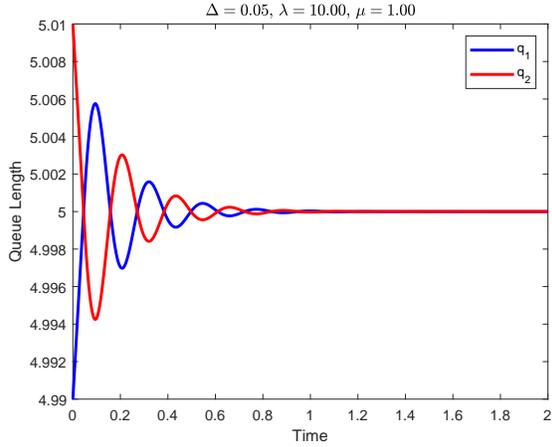} &   \includegraphics[scale=.55]{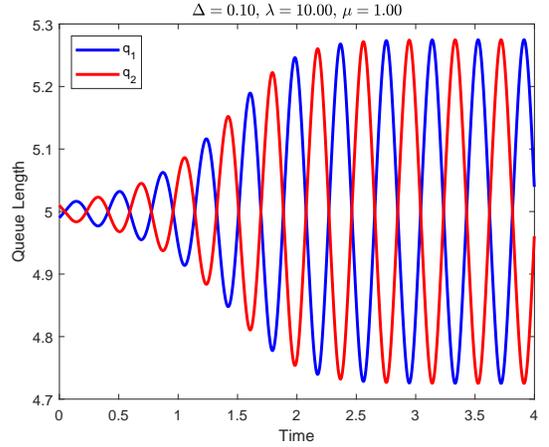} \\
(a)  & (b) \\[6pt]
 \includegraphics[scale=.55]{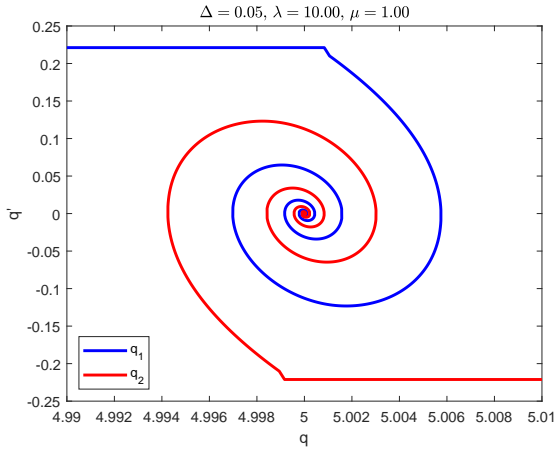} &   \includegraphics[scale=.55]{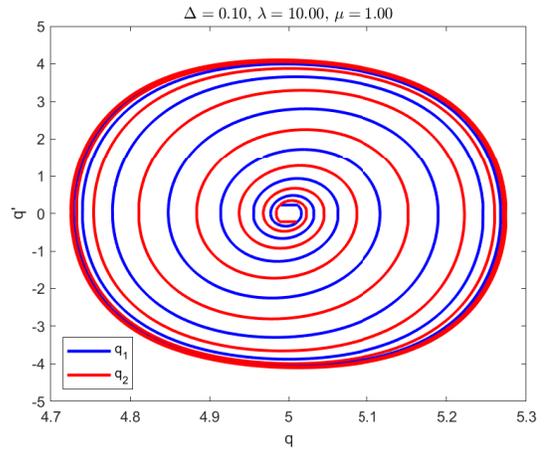} \\
(c)  & (d)  \\[6pt]
\end{tabular}
\caption{Before and after the change in stability using the choice model induced by a \textbf{normal distribution} with $\alpha = 1$ and $\sigma = 1$ with constant history function on $[-\Delta, 0]$ with $q_1 = 4.99$ and $q_2 = 5.01$, $N = 2, \lambda = 10$, $\mu = 1$. The top two plots are queue length versus time with $\Delta = .05$ (a) and $\Delta = .1$ (b). The bottom two plots are phase plots of the queue length derivative with respect to time against queue length for $\Delta = .05$ (c) and $\Delta = .1$ (d). The critical delay is $\Delta_{\text{cr}} = .0767$.}
\label{CHOICE_MODEL_fig_normal}
\end{figure}




\begin{figure}
\hspace{-10mm} \includegraphics[scale=.6]{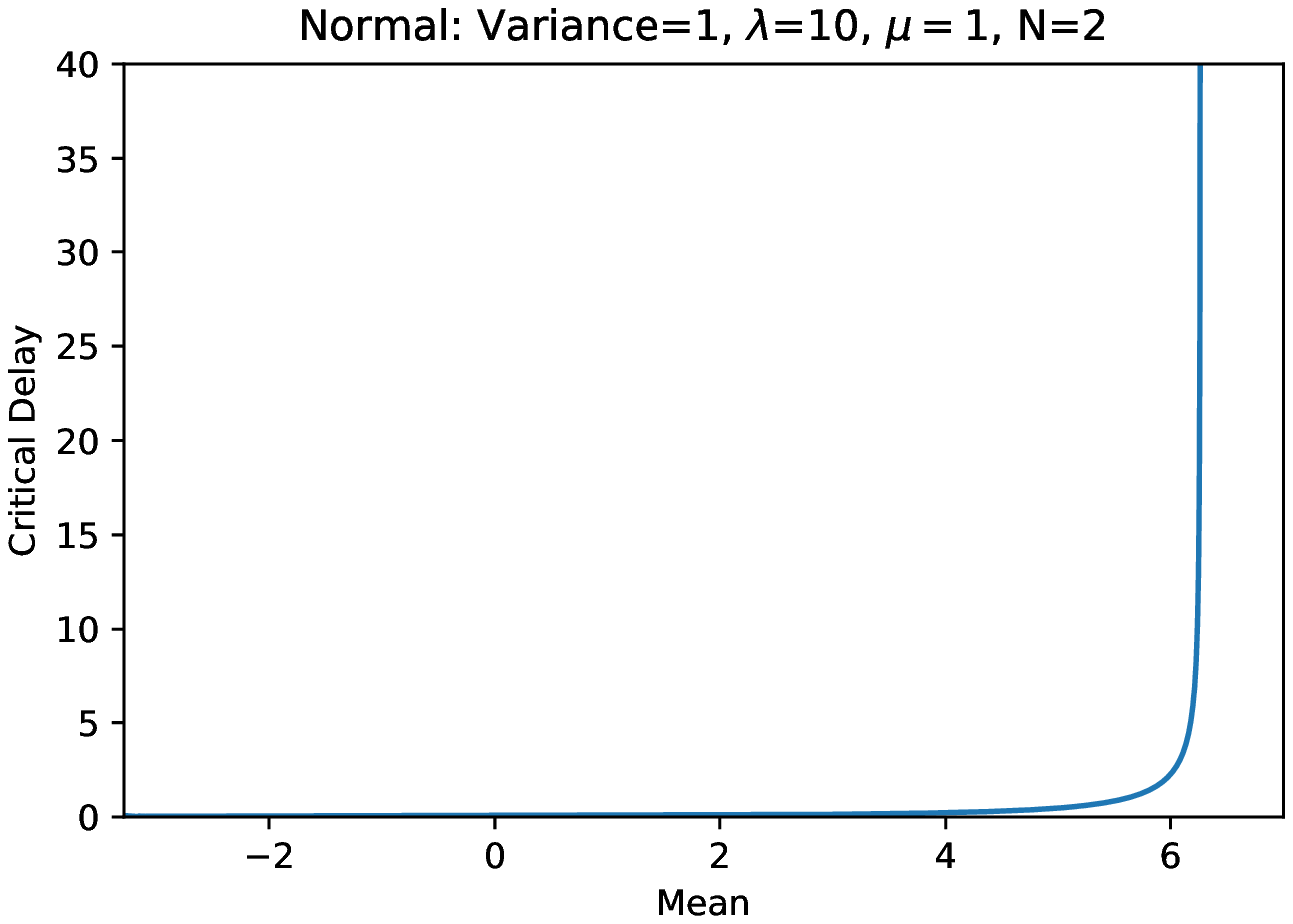} \includegraphics[scale=.6]{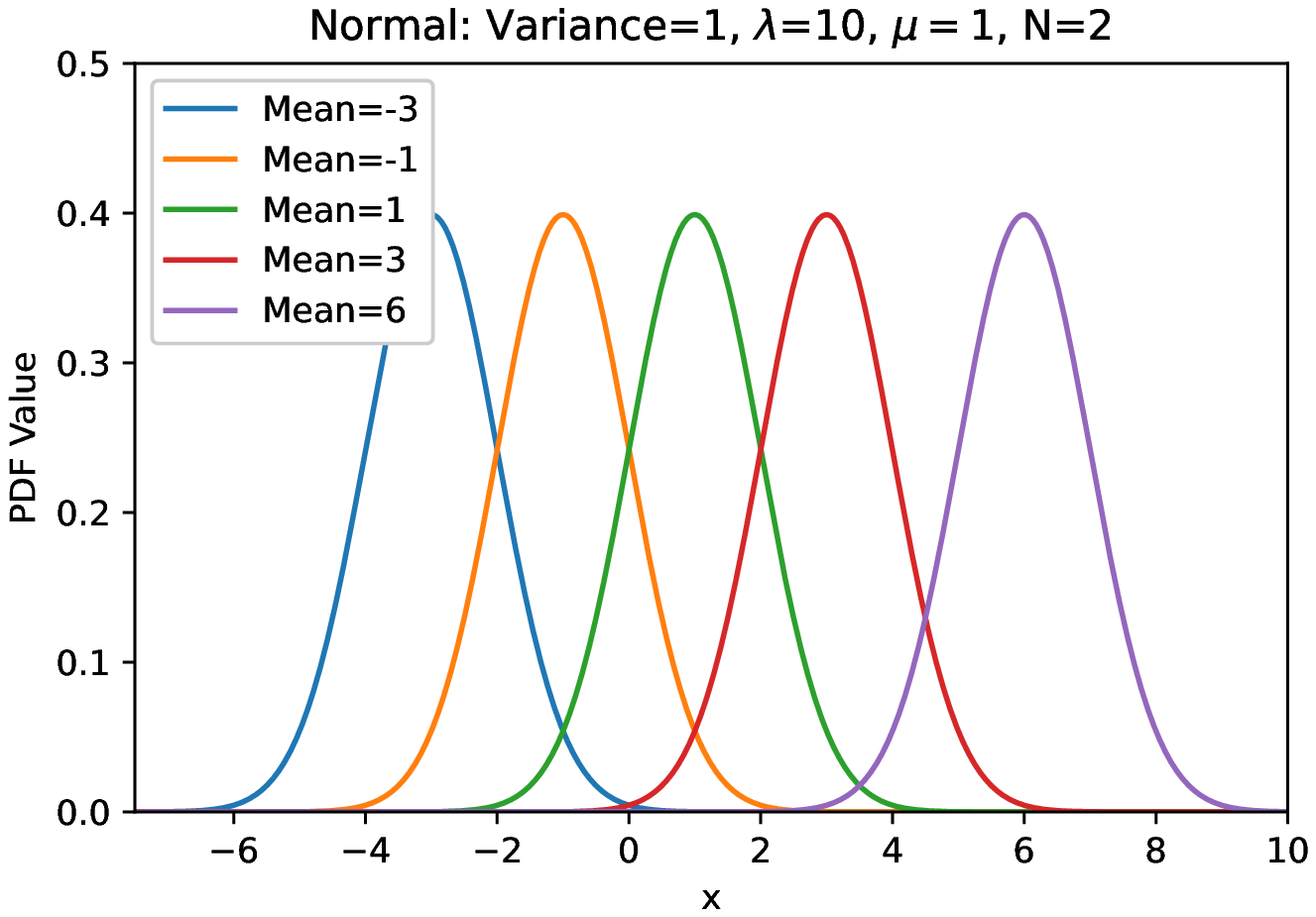}
\caption{Left: The critical delay plotted against the mean of the \textbf{normal distribution} that induces the choice model with a fixed variance of 1. Right: A plot of the probability density function used for some selected values of the mean.}
\label{CHOICE_MODEL_fig_normal_mean}
\end{figure}

\begin{figure}
\hspace{-10mm} \includegraphics[scale=.6]{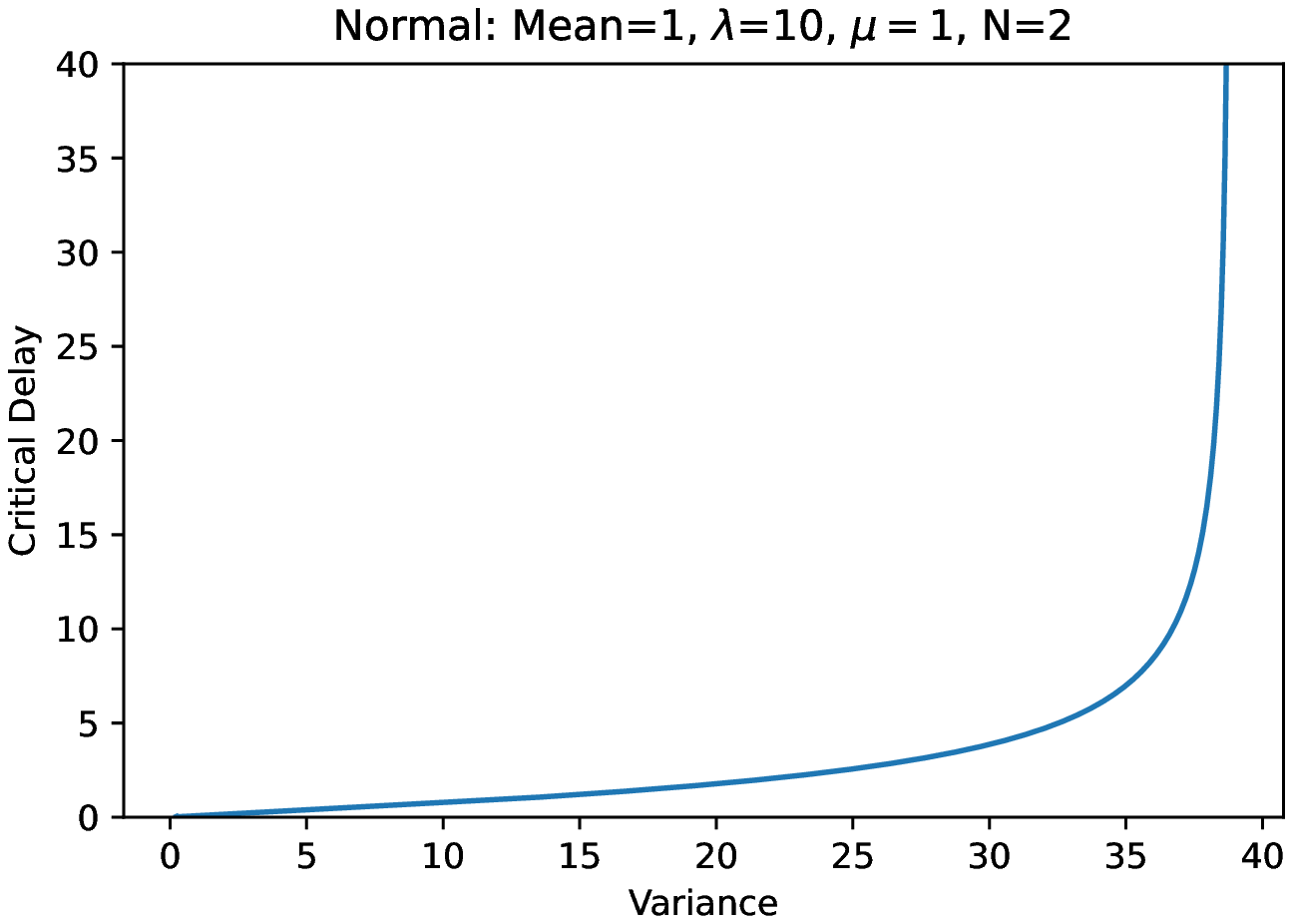} \includegraphics[scale=.6]{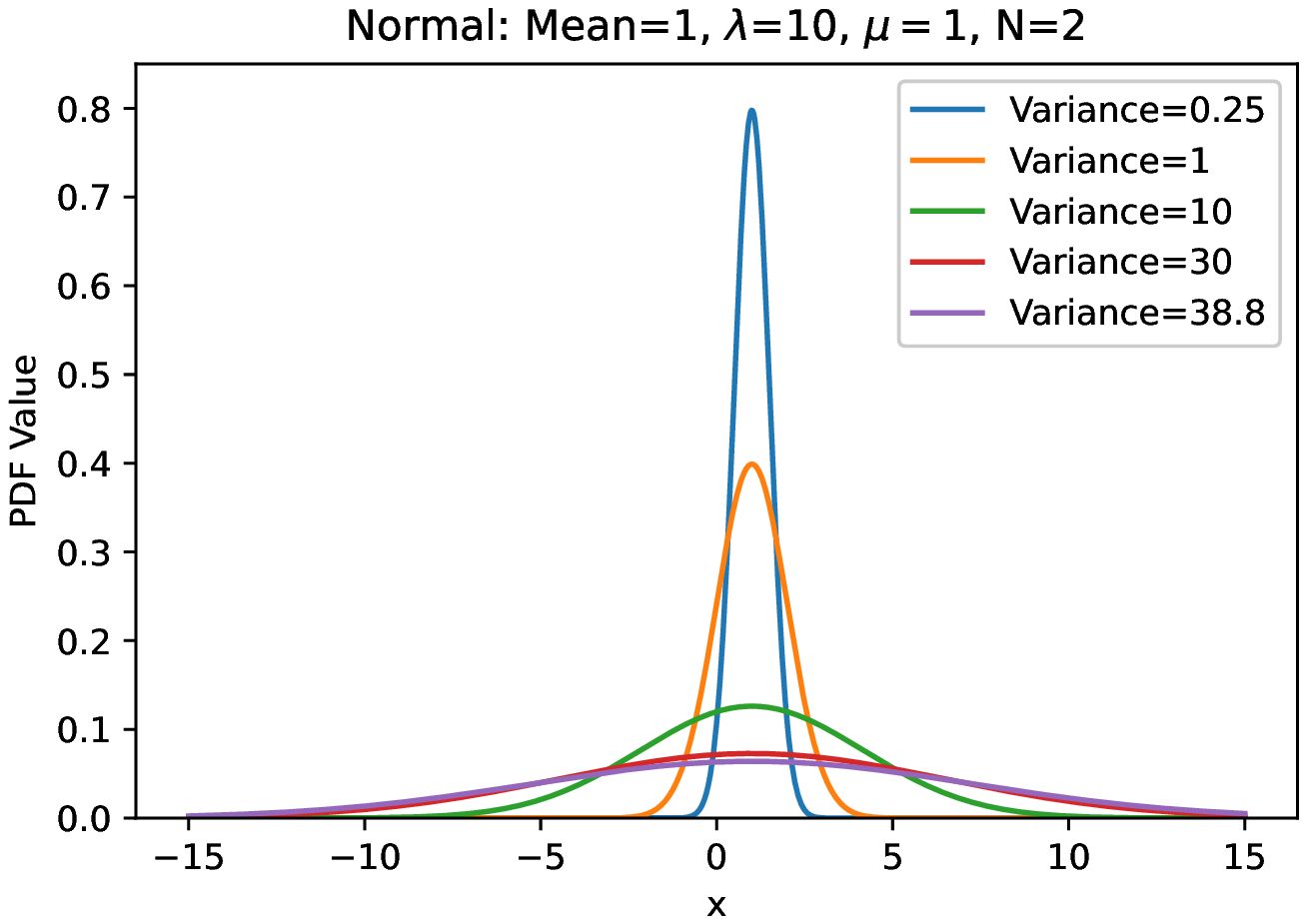}
\caption{Left: The critical delay plotted against the variance of the \textbf{normal distribution} that induces the choice model with a fixed mean of 1. Right: A plot of the probability density function used for some selected values of the variance.}
\label{CHOICE_MODEL_fig_normal_variance}
\end{figure}


\subsection{Log-Normal Distribution}

In this section, we assume that the complementary cumulative distribution function $\bar{G}$ that characterizes the choice model is given by a log-normal distribution. Below we provide some useful quantities relating to the log-normal distribution.

\begin{align}
X &\sim \text{Log-Normal}(\alpha, \sigma^2)\\
g(x) &= \frac{1}{x \sigma \sqrt{2 \pi}} \exp \left( - \frac{1}{2} \left( \frac{ \ln(x) - \alpha}{\sigma}  \right)^2  \right)\\
\bar{G}(x) &= \frac{1}{2} \left[ 1 - \text{erf} \left( \frac{\ln(x) - \alpha}{\sigma \sqrt{2}} \right)  \right]\\
\mathbb{E}[X] &= \exp \left( \alpha + \frac{\sigma^2}{2} \right)\\
\text{Var}(X) &= (\exp(\sigma^2) - 1) \exp(2\alpha + \sigma^2)\\
C &= - \sqrt{\frac{2}{\pi}}\frac{\alpha}{\sigma} \frac{\exp \left( - \frac{1}{2} \left( \frac{ \ln \left(\frac{\lambda}{N \mu} \right) - \alpha}{\sigma}  \right)^2  \right)}{\left[ 1 - \text{erf} \left( \frac{ \ln \left( \frac{\lambda}{N \mu} \right) - \alpha}{\sigma \sqrt{2}} \right)  \right] }
\end{align}

The log-normal distribution is a continuous probability distribution such that the natural logarithm applied to a log-normal random variable with parameters $\alpha$ and $\sigma$ results in a random variable that has a normal distribution with mean $\alpha$ and standard deviation $\sigma$. In Figure \ref{CHOICE_MODEL_fig_lognormal} we show queue length plots and phase diagrams on each side of the critical delay. In Figure \ref{CHOICE_MODEL_fig_lognormal_mean} we see how the value of the critical delay varies as the mean is varied with a fixed variance and in Figure \ref{CHOICE_MODEL_fig_lognormal_variance} we see how it changes as the variance is varied with a fixed mean. As the mean varies, there is a critical value of the mean such that the critical delay can be made arbitrarily large when approached. This is similar to the case shown for the normal distribution, but it differs in that we see that in the case of the log-normal distribution, the critical delay gets larger as the mean approaches zero as well. Additionally, the log-normal distribution appears to have a larger critical variance than the normal distribution does for unit mean, but the critical delay is larger in the log-normal case for small variance than it is in the normal case.

\begin{figure}
\begin{tabular}{cc}
  \includegraphics[scale=.55]{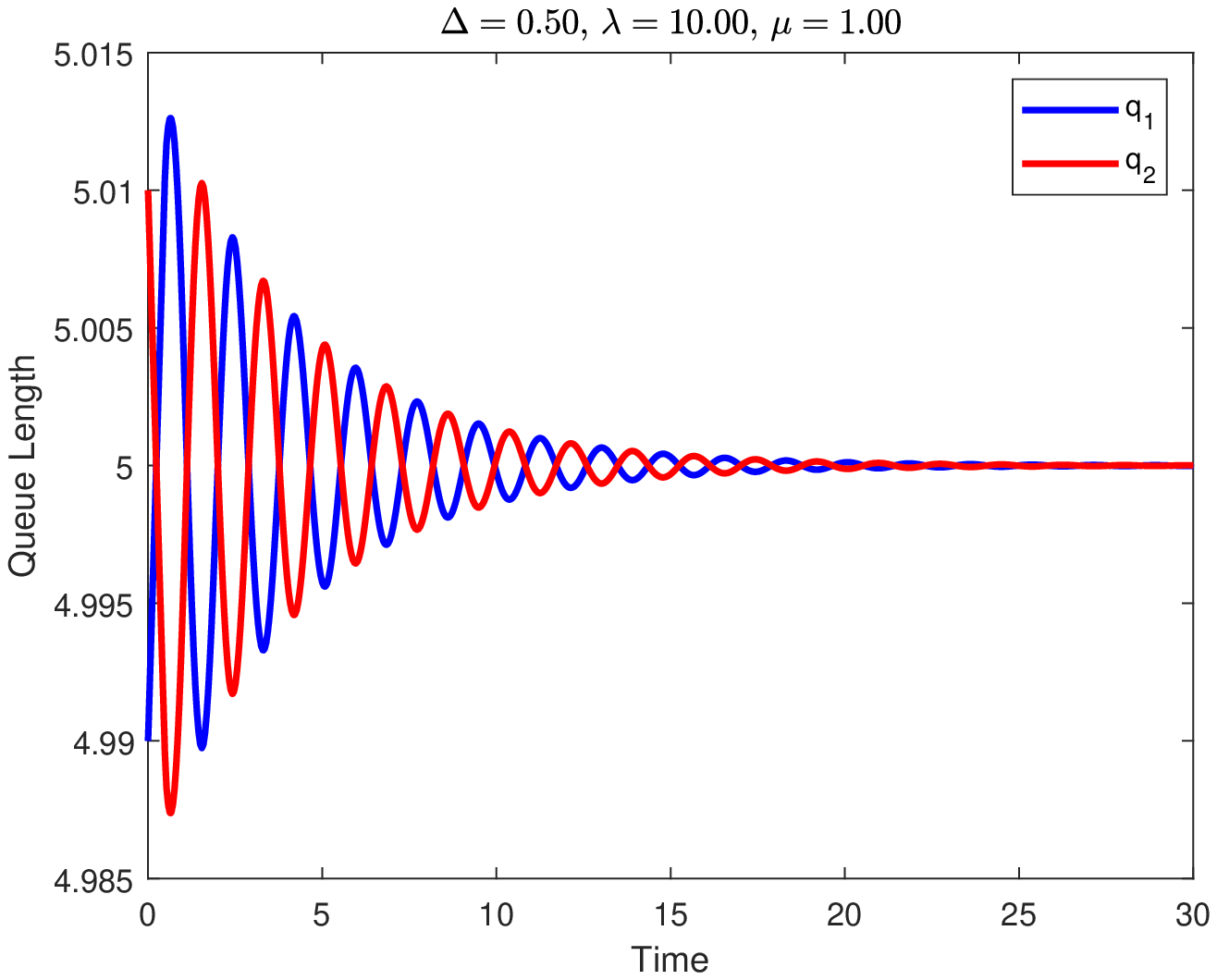} &   \includegraphics[scale=.55]{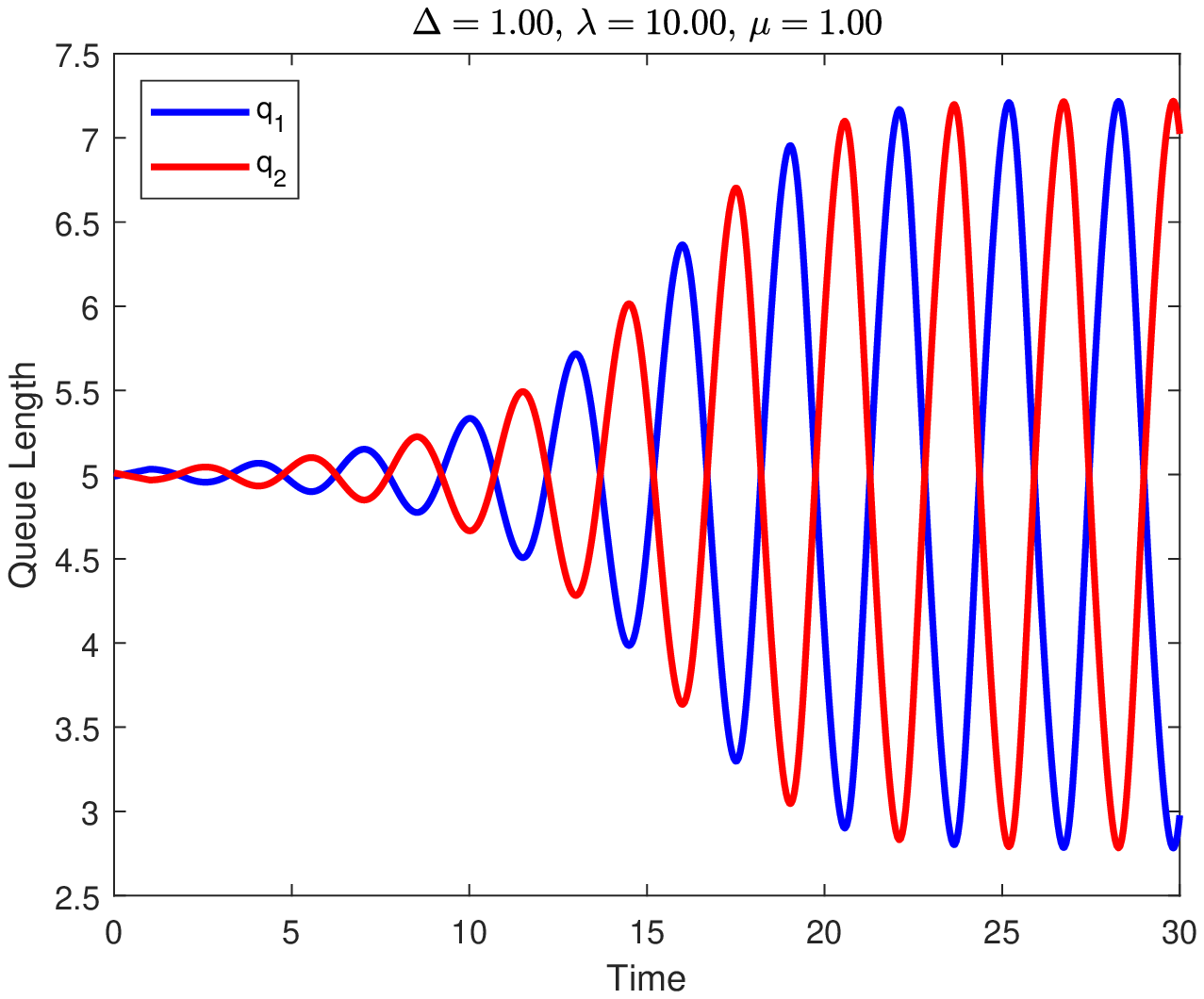} \\
(a)  & (b) \\[6pt]
 \includegraphics[scale=.55]{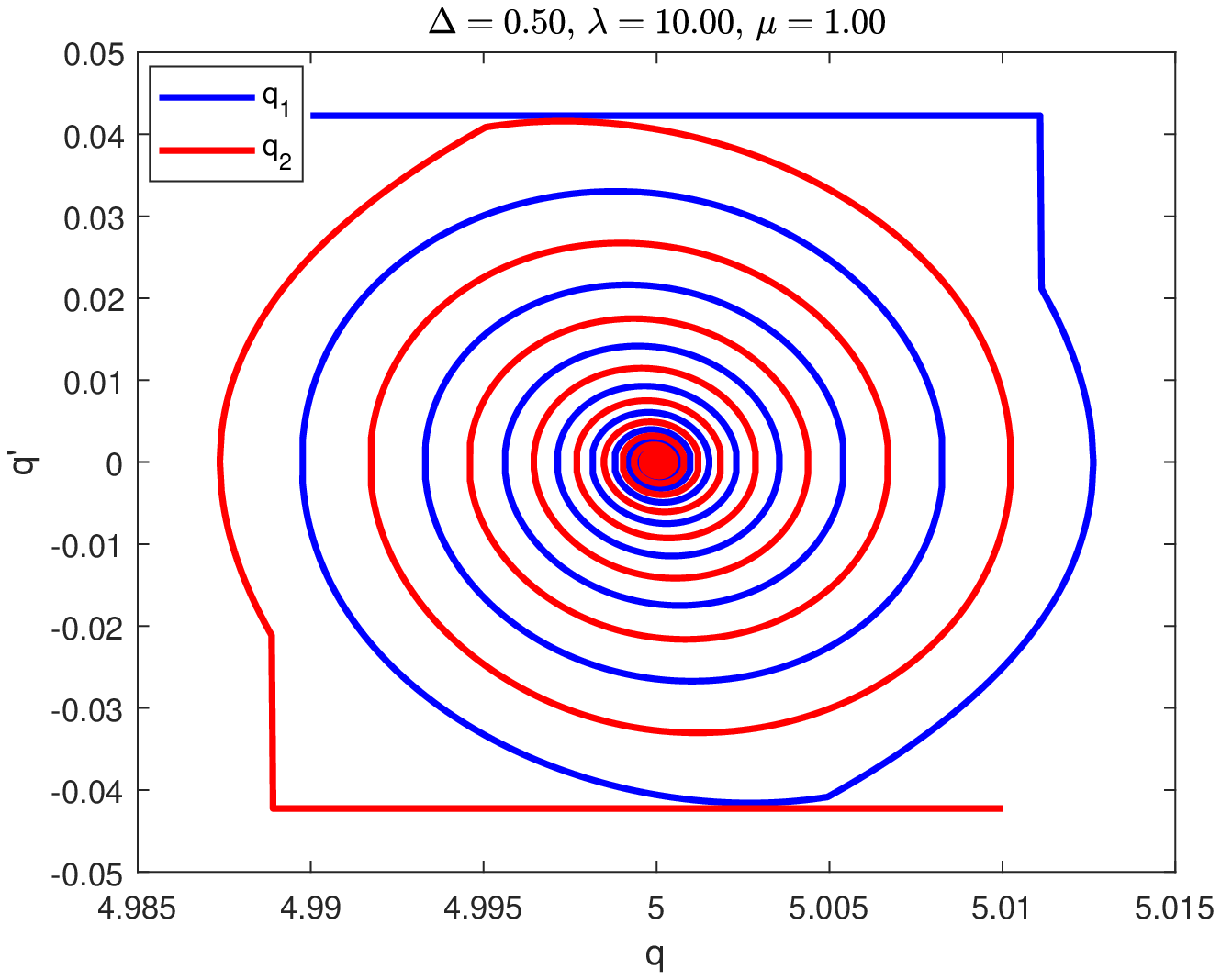} &   \includegraphics[scale=.55]{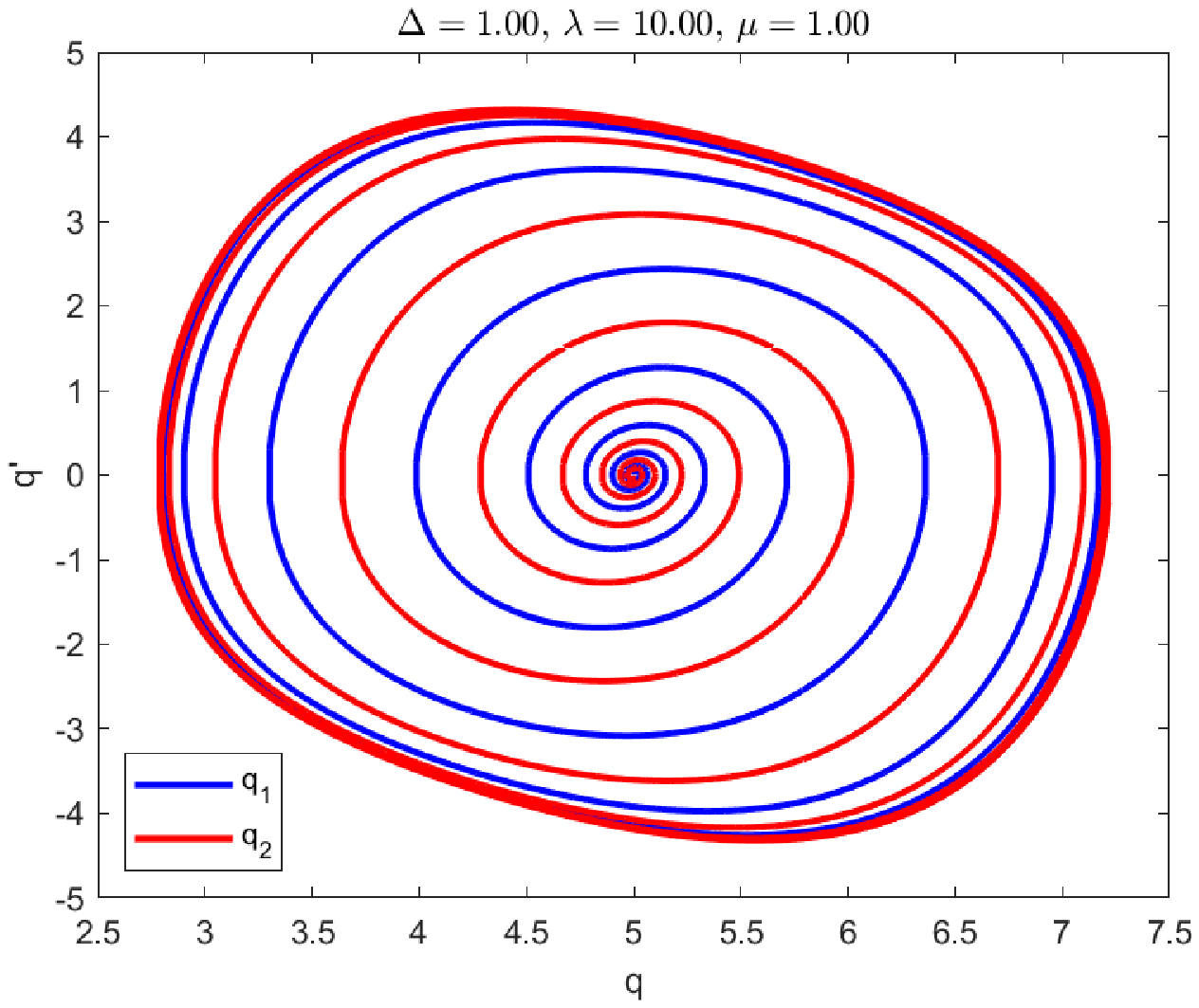} \\
(c)  & (d)  \\[6pt]
\end{tabular}
\caption{Before and after the change in stability using the choice model induced by a \textbf{log-normal distribution} with $\alpha = - \frac{1}{2} \log(2)$ and $ \sigma = \sqrt{\log(2)}$ (which results in mean 1 and variance 1) with constant history function on $[-\Delta, 0]$ with $q_1 = 4.99$ and $q_2 = 5.01$, $N = 2, \lambda = 10$, $\mu = 1$. The left two plots are queue length versus time with $\Delta = .5$ (Left) and $\Delta = 1$ (Right). The right two plots are phase plots of the queue length derivative with respect to time against queue length for $\Delta = .5$ (Left) and $\Delta = 1$ (Right). The critical delay is $\Delta_{\text{cr}} = .6148$.}
\label{CHOICE_MODEL_fig_lognormal}
\end{figure}

\begin{figure}
\hspace{-10mm} \includegraphics[scale=.6]{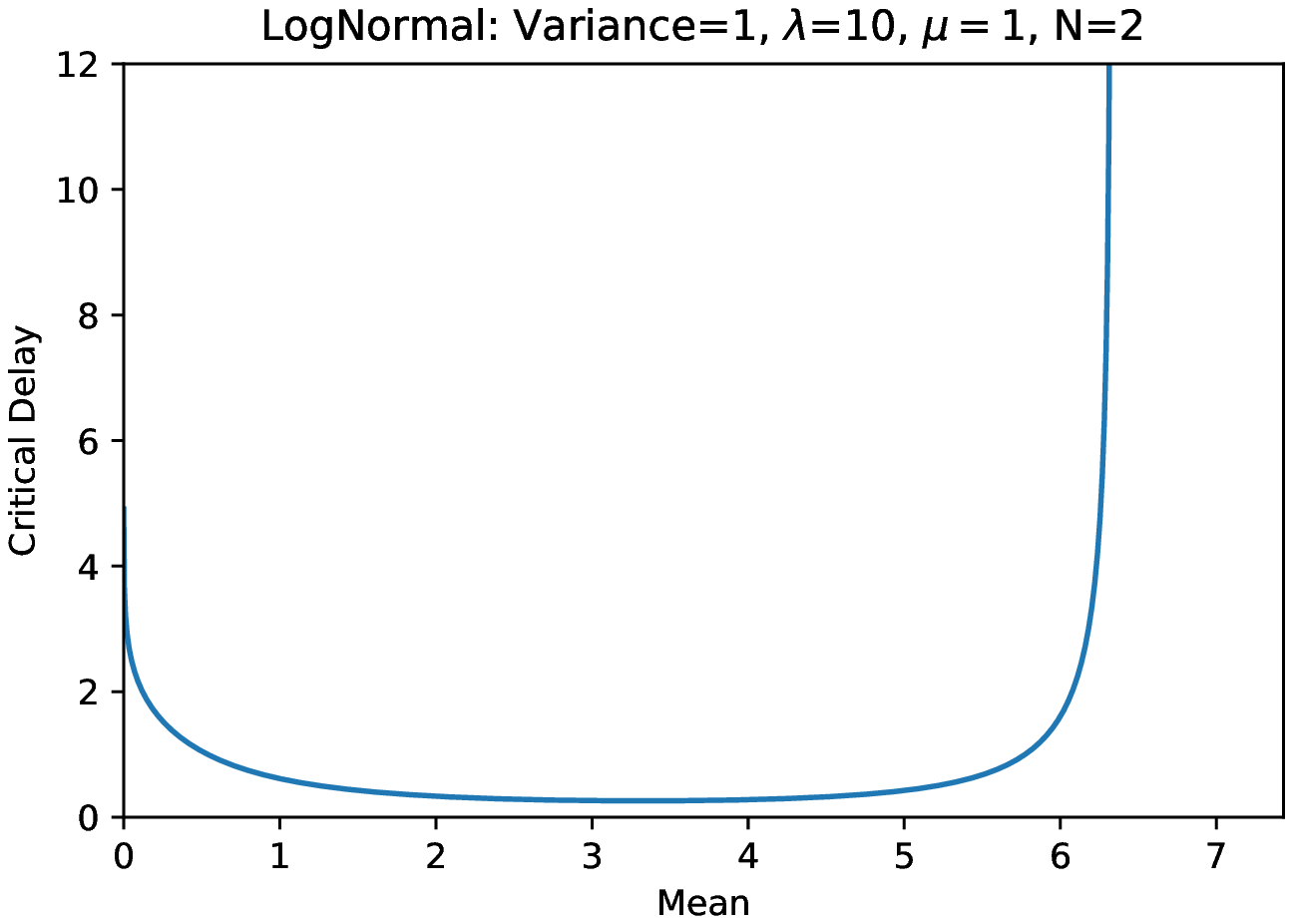} \includegraphics[scale=.6]{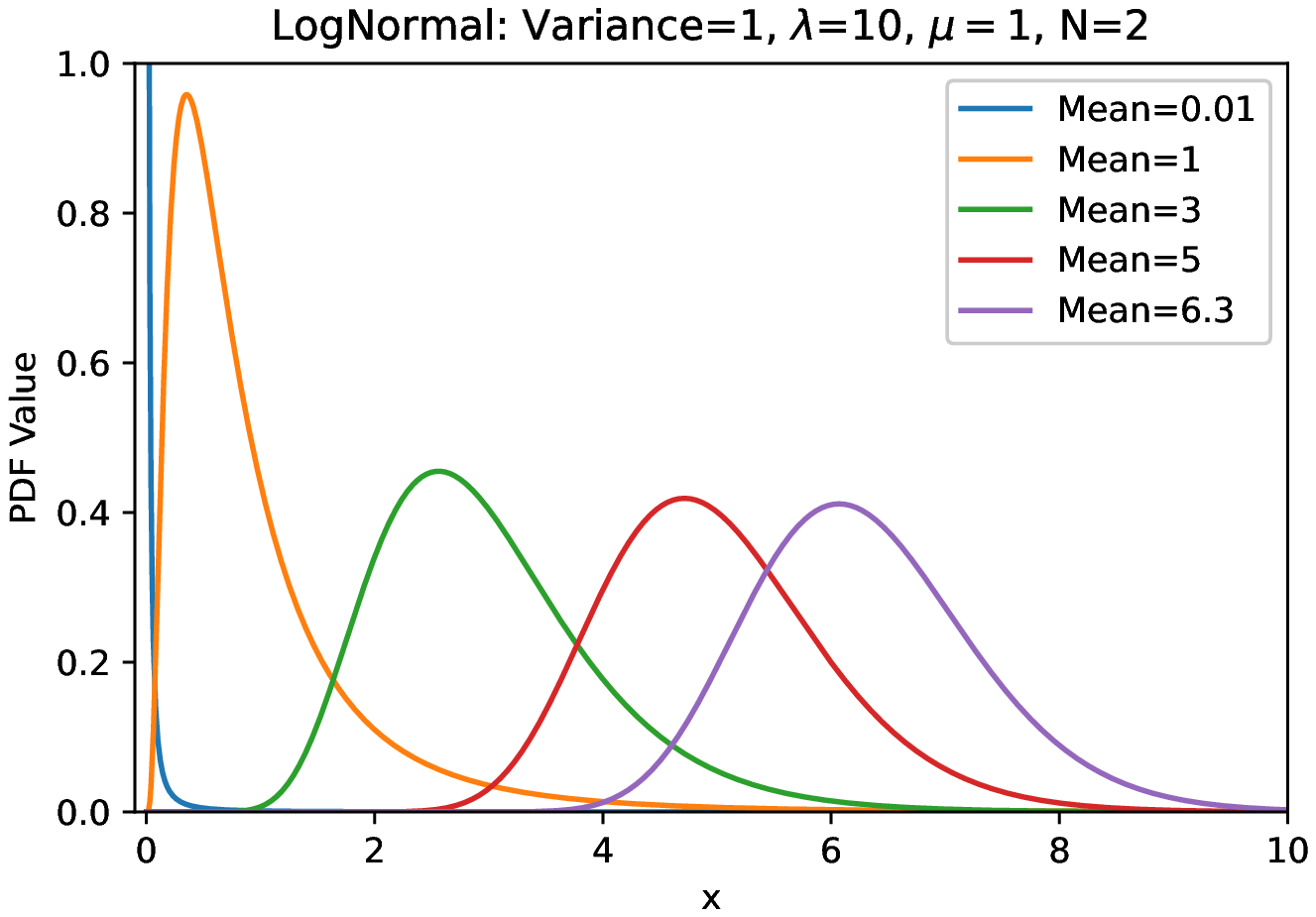}
\caption{Left: The critical delay plotted against the mean of the \textbf{log-normal distribution} that induces the choice model with a fixed variance of 1. Right: A plot of the probability density function used for some selected values of the mean.}
\label{CHOICE_MODEL_fig_lognormal_mean}
\end{figure}

\begin{figure}
\hspace{-10mm} \includegraphics[scale=.6]{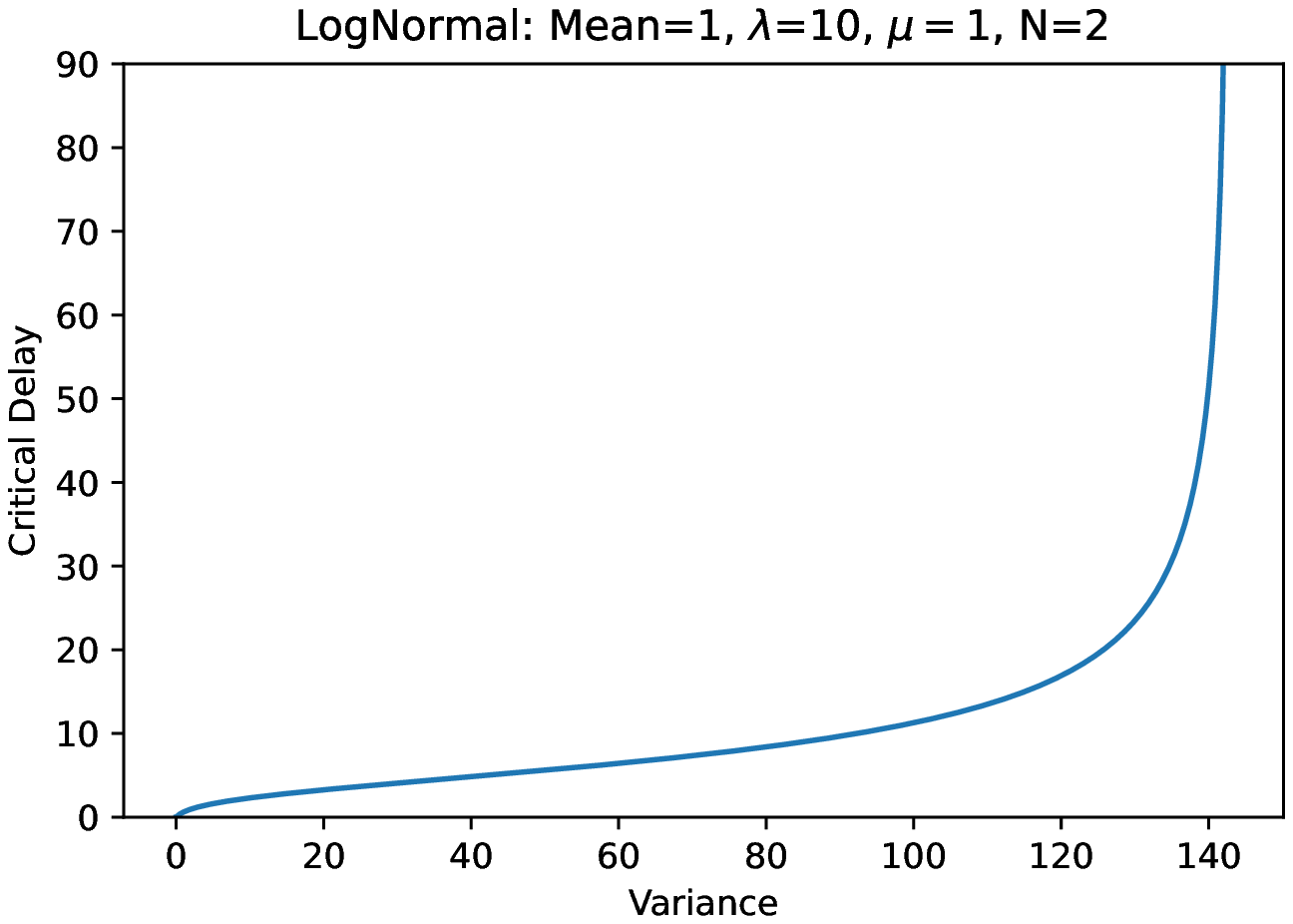} \includegraphics[scale=.6]{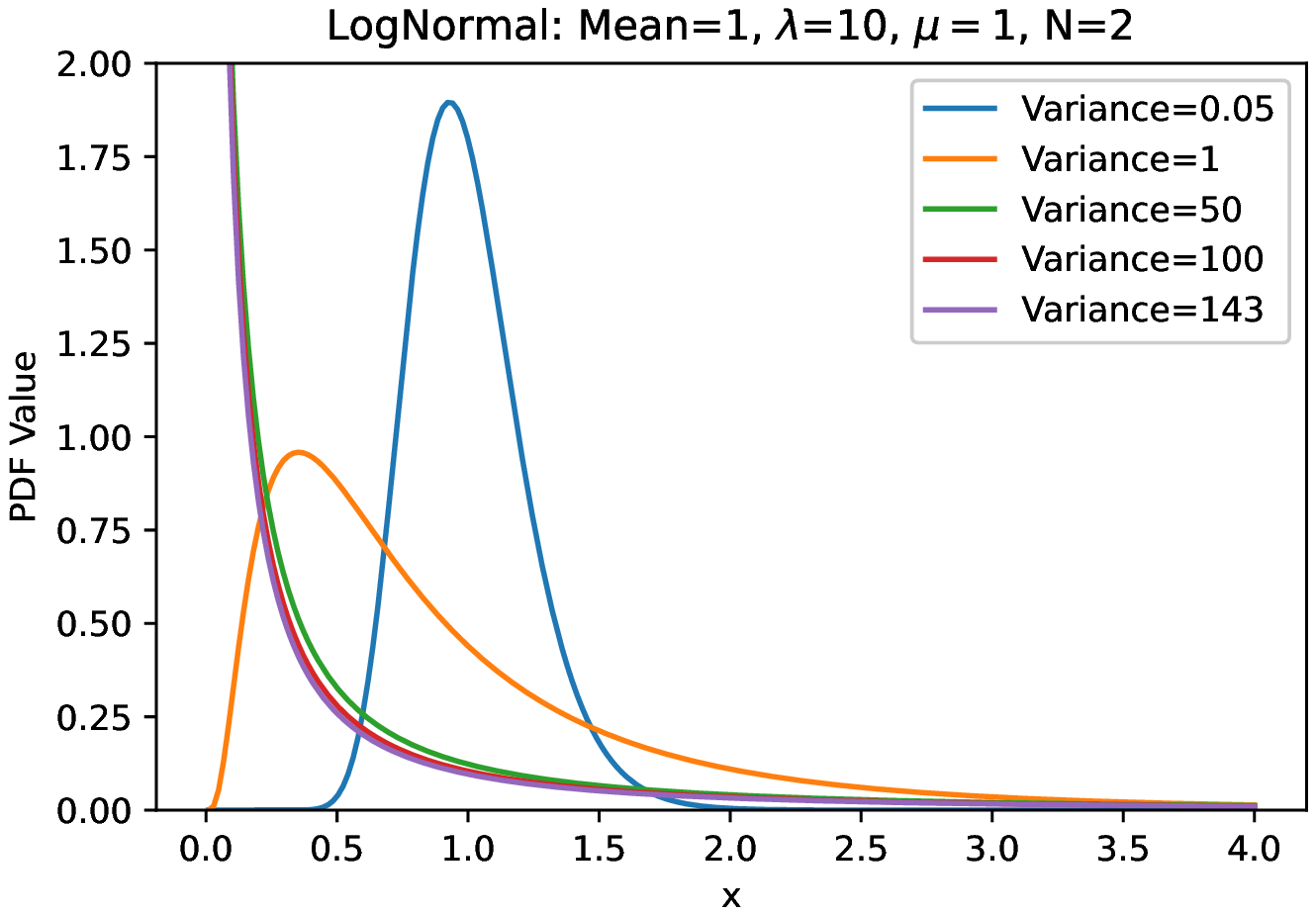}
\caption{Left: The critical delay plotted against the variance of the \textbf{log-normal distribution} that induces the choice model with a fixed mean of 1. Right: A plot of the probability density function used for some selected values of the variance.}
\label{CHOICE_MODEL_fig_lognormal_variance}
\end{figure}


\subsection{Weibull Distribution}

In this section, we assume that the complementary cumulative distribution function $\bar{G}$ that characterizes the choice model is given by a Weibull distribution. Below we provide some useful quantities relating to the Weibull distribution.

\begin{align}
X &\sim \text{Weibull}(\alpha, \beta)\\
g(x) &= \beta \alpha x^{\alpha-1}e^{-\beta x^\alpha}, \hspace{5mm} x, \alpha, \beta > 0 \\
\bar{G}(x) &= e^{-\beta x^\alpha}\\
\mathbb{E}[X] &= \beta^{-1/\alpha}\Gamma(1+1/\alpha)\\
\text{Var}(X) &= \frac{1}{\beta^{2\alpha}}\left[\Gamma\left(1+\frac{2}{\alpha}\right) - \left(\Gamma\left(1+\frac{1}{\alpha}\right)\right)^2\right]\\
C &= - \frac{\lambda \beta \alpha \left( \frac{\lambda}{\mu N}\right)^{\alpha-1} }{N } 
\end{align}

The Weibull distribution is a continuous probability distribution with two positive parameters $\alpha$ and $\beta$. Some of the above quantities are defined in terms of the gamma function \begin{eqnarray}
\Gamma(x) := \int_{0}^{\infty} z^{x-1} e^{-z} dz, \hspace{5mm} \text{Re}(x) > 0
\label{CHOICE_MODEL_gamma_function_def}
\end{eqnarray} which has the property that $\Gamma(n) = (n-1)!$ when $n$ is a positive integer and can be viewed as a smooth interpolation of the factorial function. The Weibull distribution has a polynomial hazard function \begin{eqnarray}
h(x) = \beta \alpha x^{\alpha - 1}
\end{eqnarray} so that when $\alpha < 1$ it is a decreasing function and when $\alpha > 1$ it is an increasing function. We note that when $\alpha = 1$, the hazard rate is a constant $\beta$, just like that of the exponential distribution. Furthermore, the Weibull distribution reduces to an exponential distribution with parameter $\beta > 0$ when $\alpha = 1$.

In Figure \ref{CHOICE_MODEL_fig_weibull} we show queue length plots and phase diagrams on each side of the critical delay. In Figures \ref{CHOICE_MODEL_fig_weibull_mean} and \ref{CHOICE_MODEL_fig_weibull_mean_a_point5} we see how the value of the critical delay varies as the mean is varied with a fixed variance and in Figures \ref{CHOICE_MODEL_fig_weibull_variance} and \ref{CHOICE_MODEL_fig_weibull_variance_a_point5} we see how it changes as the variance is varied with a fixed mean. We see that in the cases considered, there is a critical mean value and a critical variance value at which the critical delay is unbounded. The shape of the critical delay against variance plot appears similar to the one in the log-normal case, noting the initial concave down section for small variance.

\begin{figure}
\begin{tabular}{cc}
  \includegraphics[scale=.55]{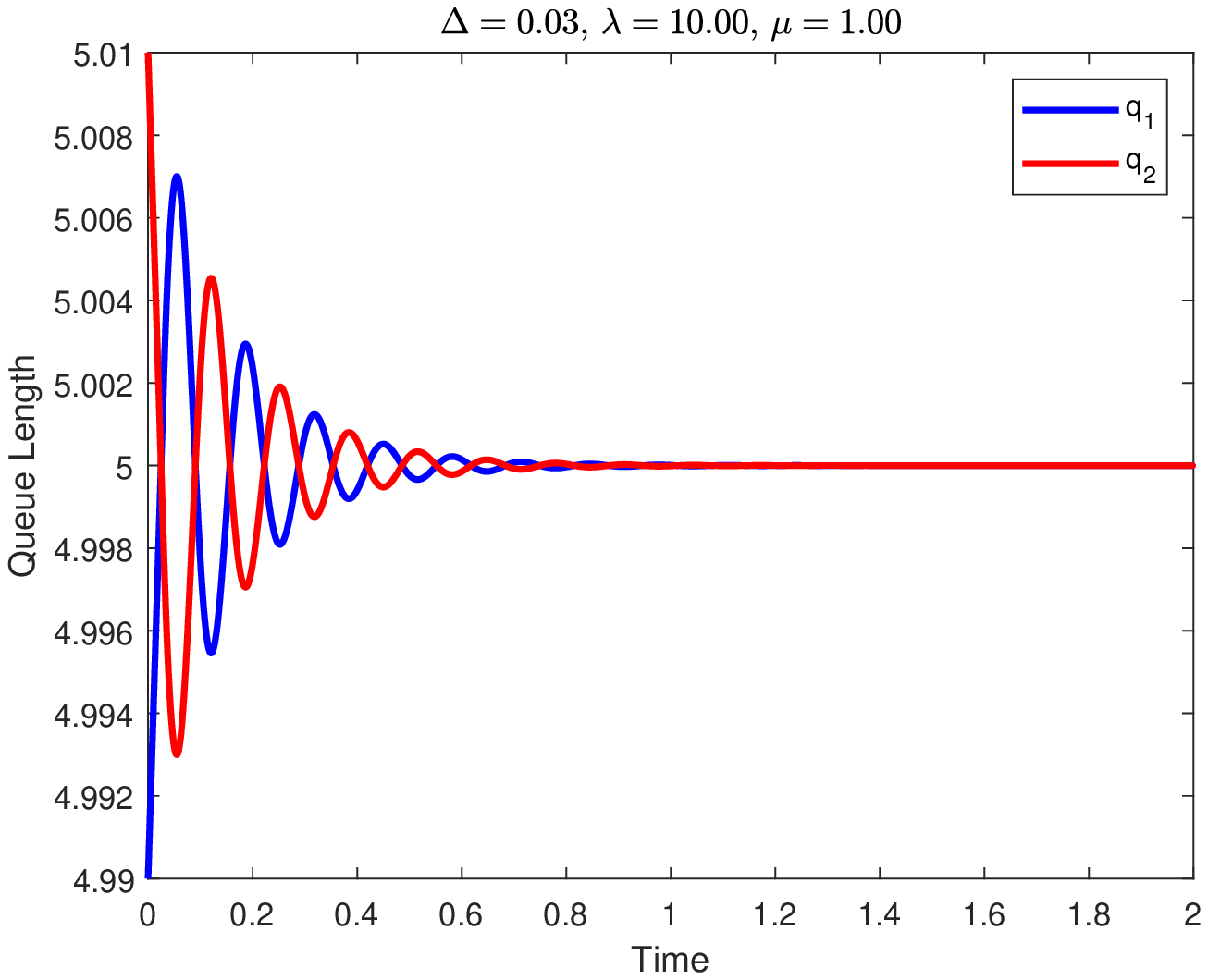} &   \includegraphics[scale=.55]{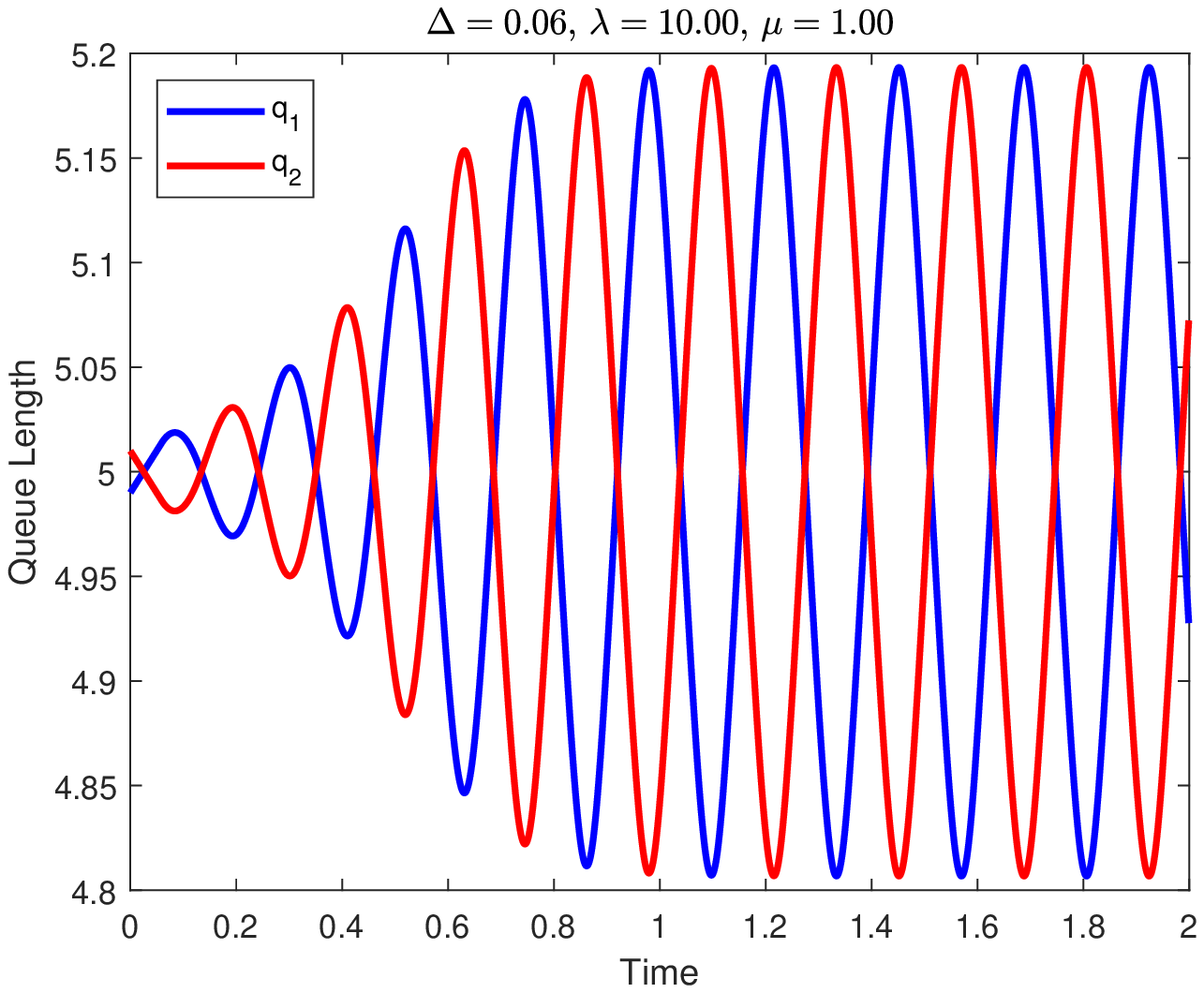} \\
(a)  & (b) \\[6pt]
 \includegraphics[scale=.55]{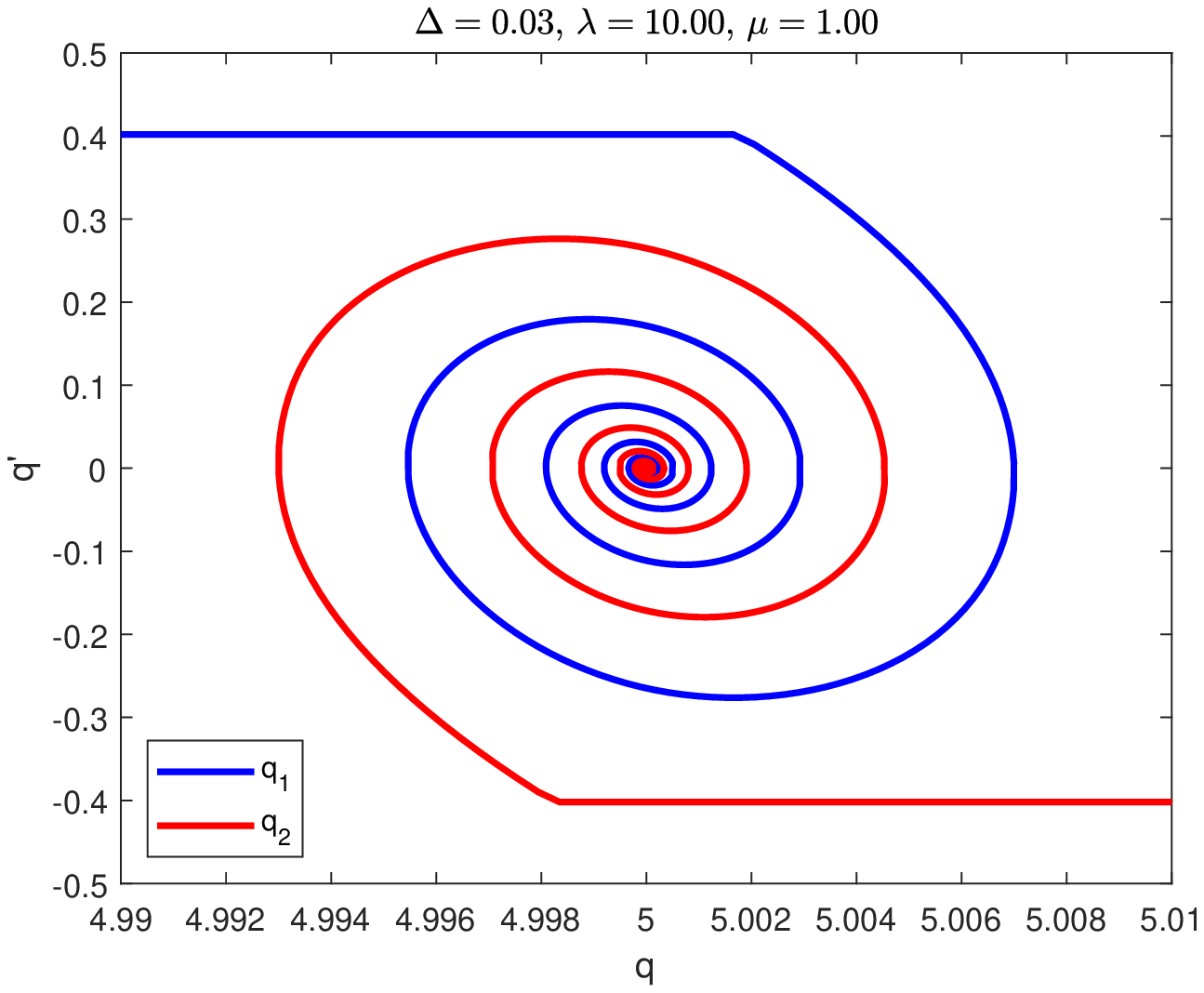} &   \includegraphics[scale=.55]{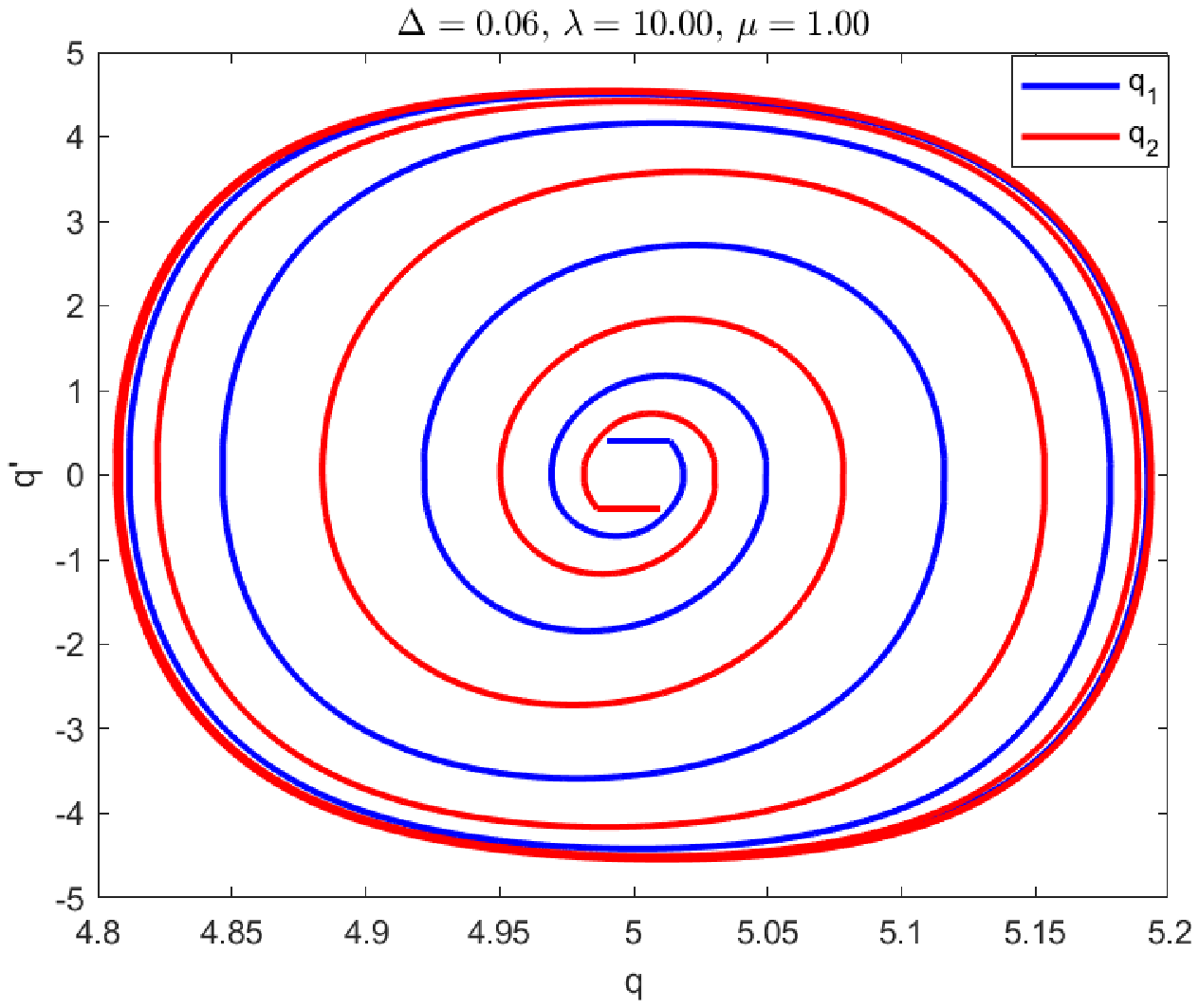} \\
(c)  & (d)  \\[6pt]
\end{tabular}
\caption{Before and after the change in stability using the choice model induced by a \textbf{Weibull distribution} with $\alpha = 2$ and $\beta = \left( \Gamma(1 + \frac{1}{\alpha}) \right)^{\alpha}$ with constant history function on $[-\Delta, 0]$ with $q_1 = 4.99$ and $q_2 = 5.01$, $N = 2, \lambda = 10$, $\mu = 1$. The top two plots are queue length versus time with $\Delta = .03$ (a) and $\Delta = .06$ (b). The bottom two plots are phase plots of the queue length derivative with respect to time against queue length for $\Delta = .03$ (c) and $\Delta = .06$ (d). The critical delay is $\Delta_{\text{cr}} = .0407$.}
\label{CHOICE_MODEL_fig_weibull}
\end{figure}




\begin{figure}
\hspace{-10mm} \includegraphics[scale=.6]{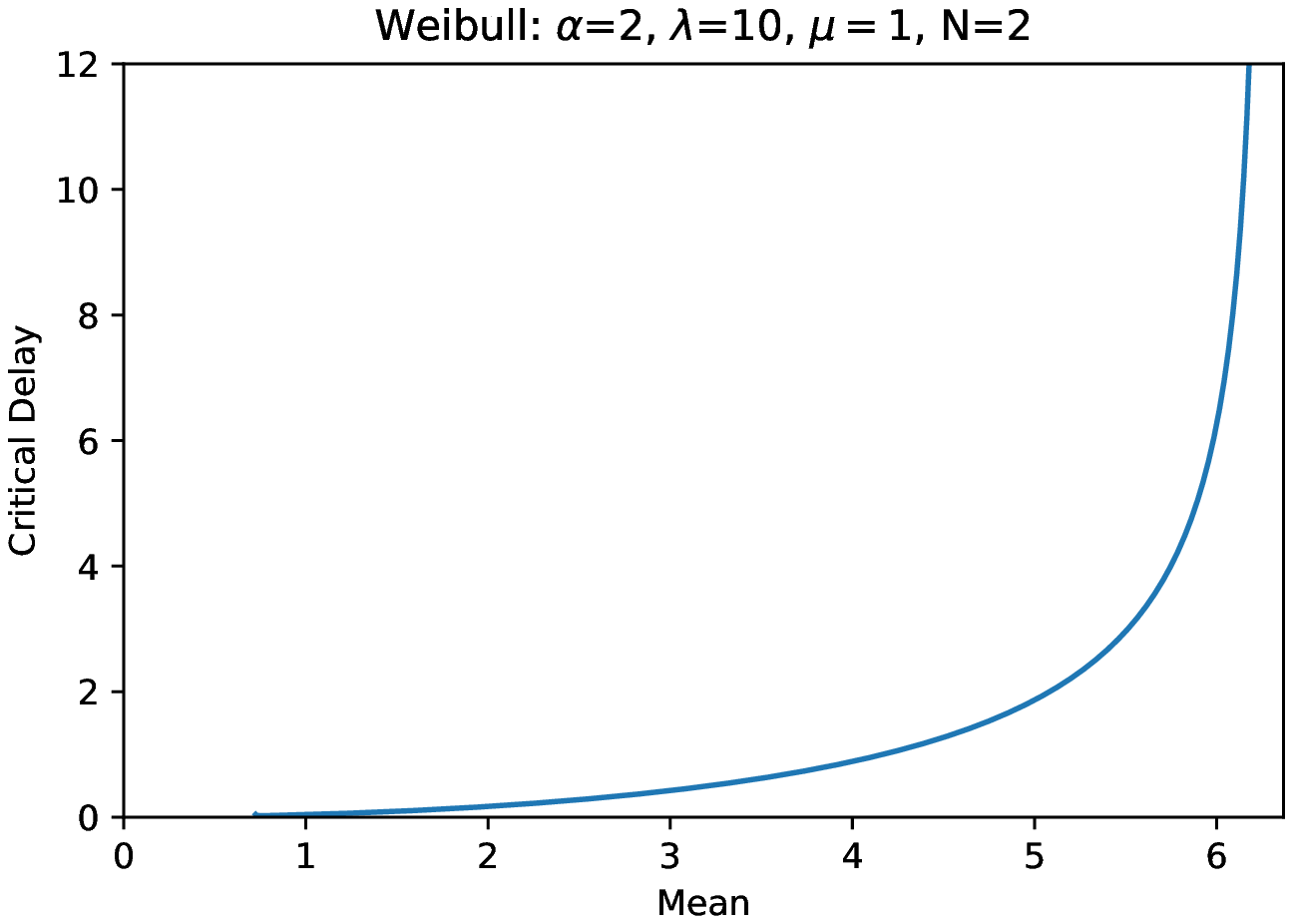} \includegraphics[scale=.6]{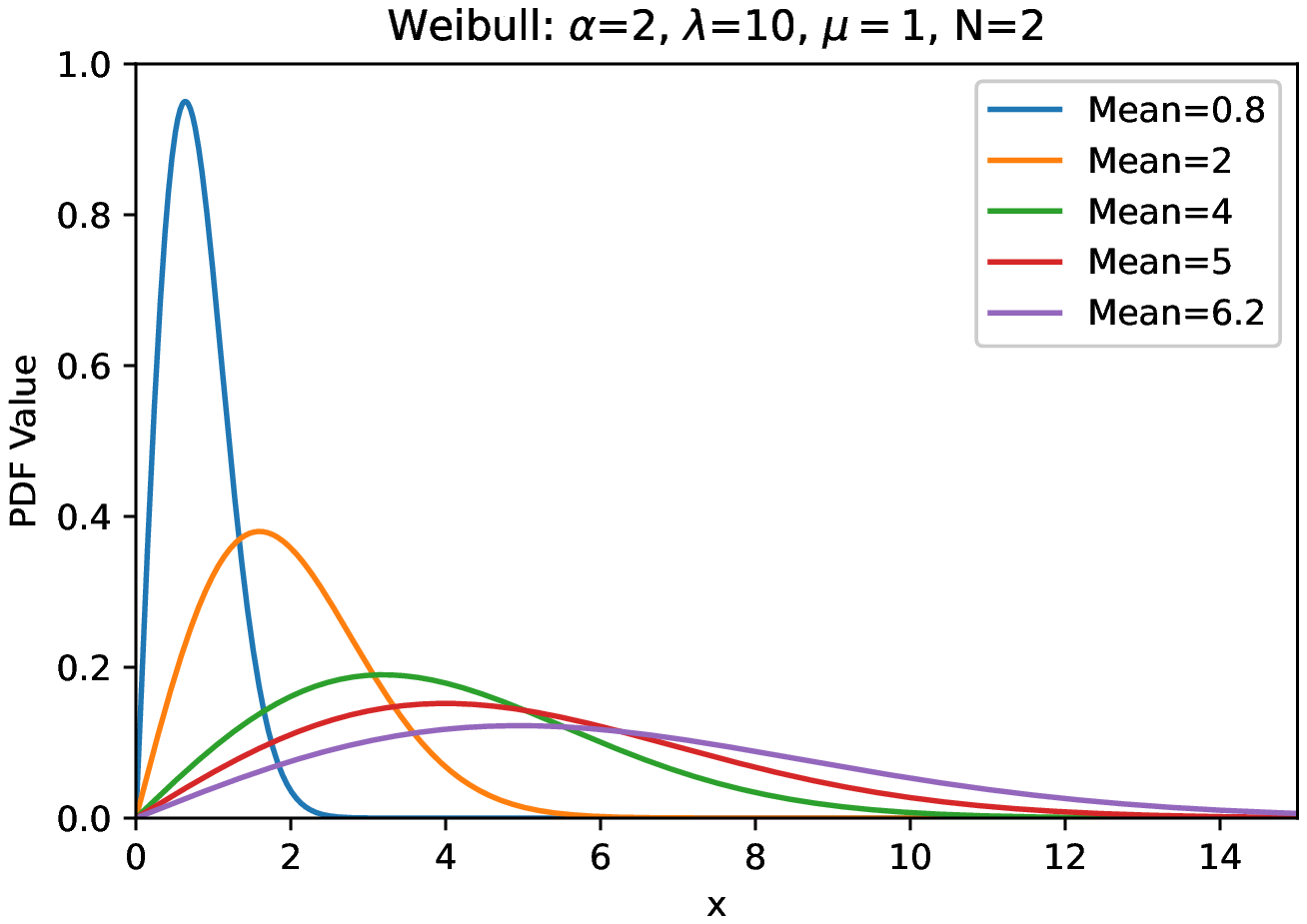}
\caption{Left: The critical delay plotted against the mean of the \textbf{Weibull distribution} that induces the choice model with fixed $\alpha = 2$. Right: A plot of the probability density function used for some selected values of the mean.}
\label{CHOICE_MODEL_fig_weibull_mean}
\end{figure}

\begin{figure}
\hspace{-10mm} \includegraphics[scale=.6]{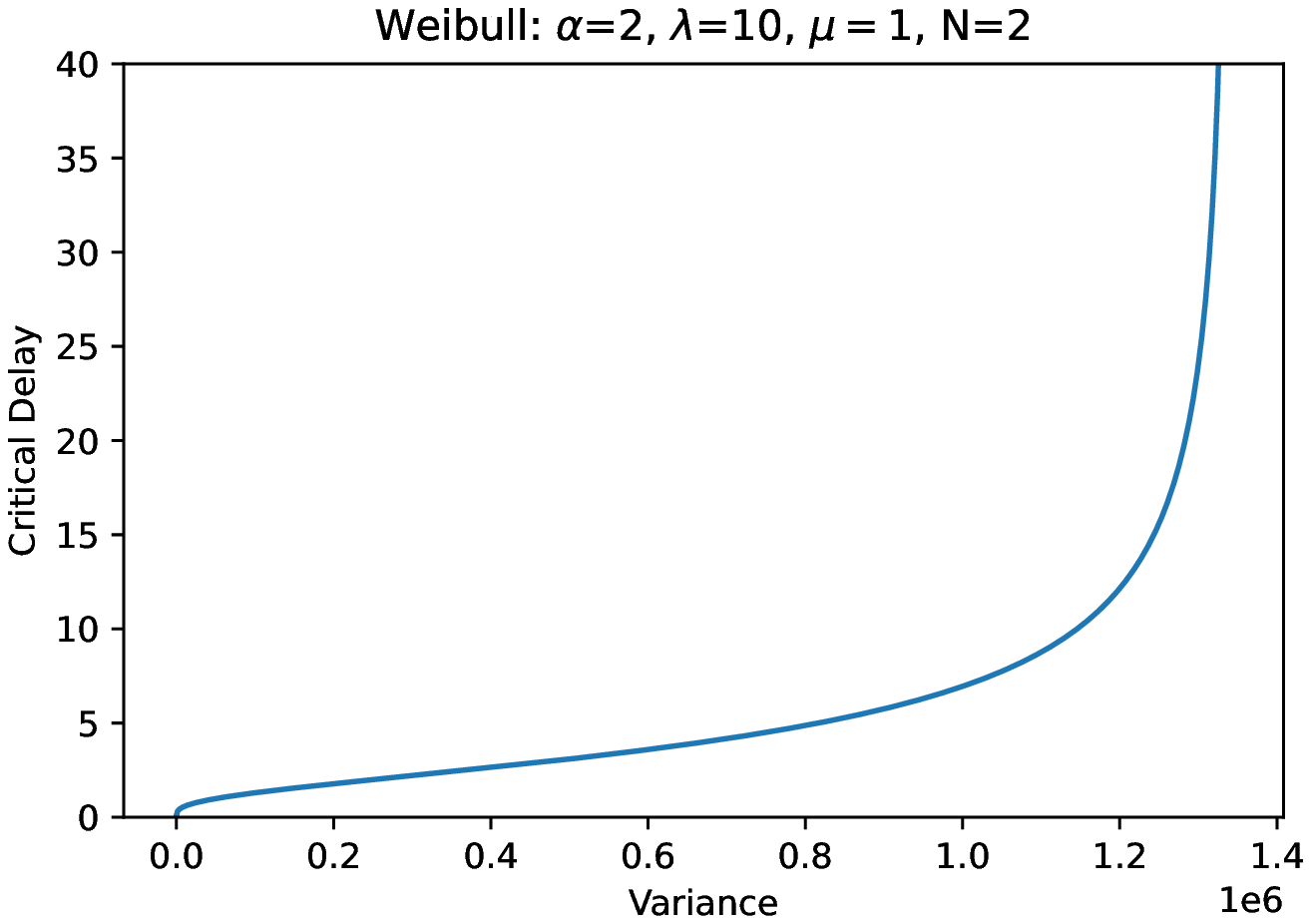} \includegraphics[scale=.6]{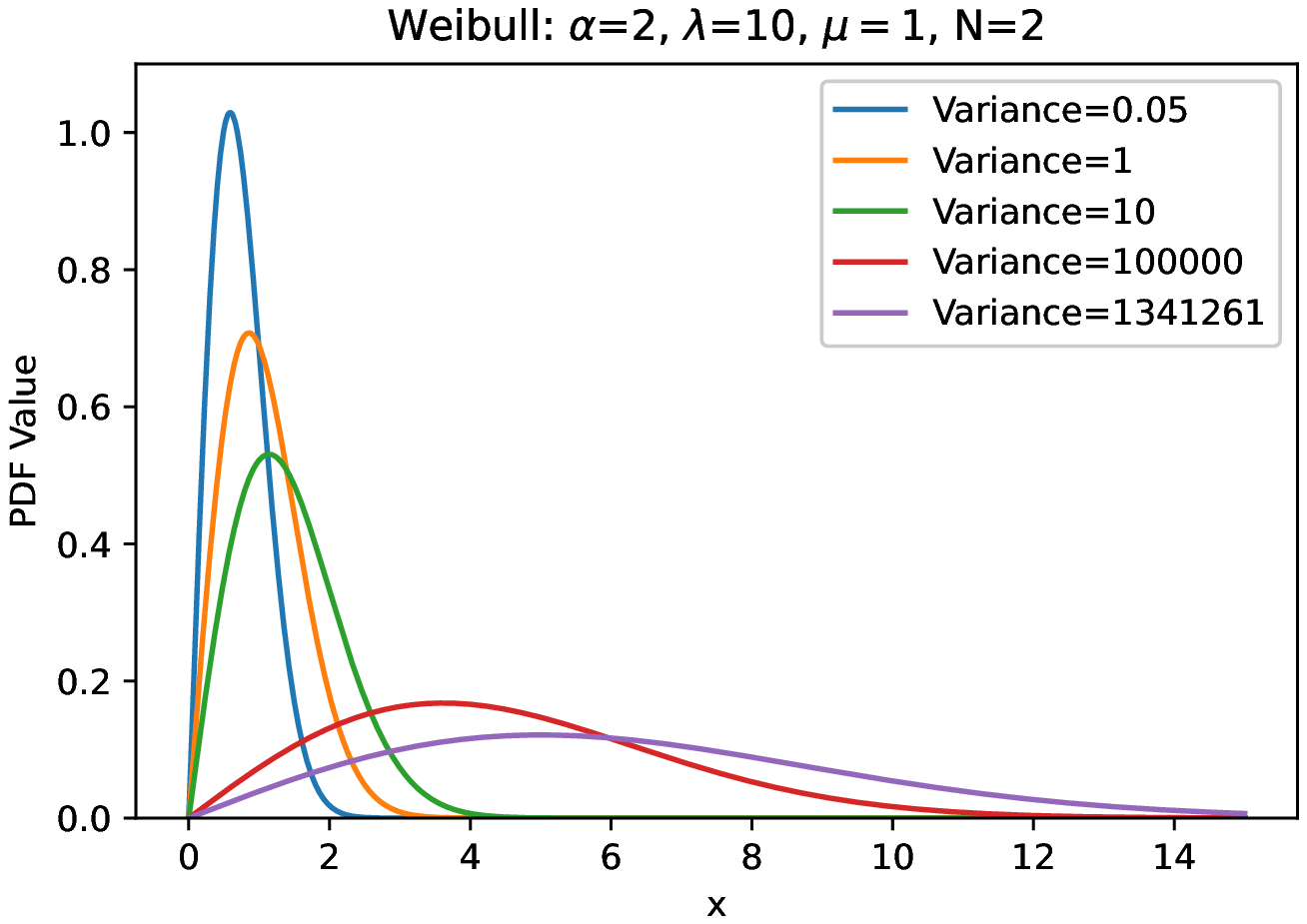}
\caption{Left: The critical delay plotted against the variance of the \textbf{Weibull distribution} that induces the choice model with fixed $\alpha = 2$. Right: A plot of the probability density function used for some selected values of the variance.}
\label{CHOICE_MODEL_fig_weibull_variance}
\end{figure}



\begin{figure}
\hspace{-10mm} \includegraphics[scale=.6]{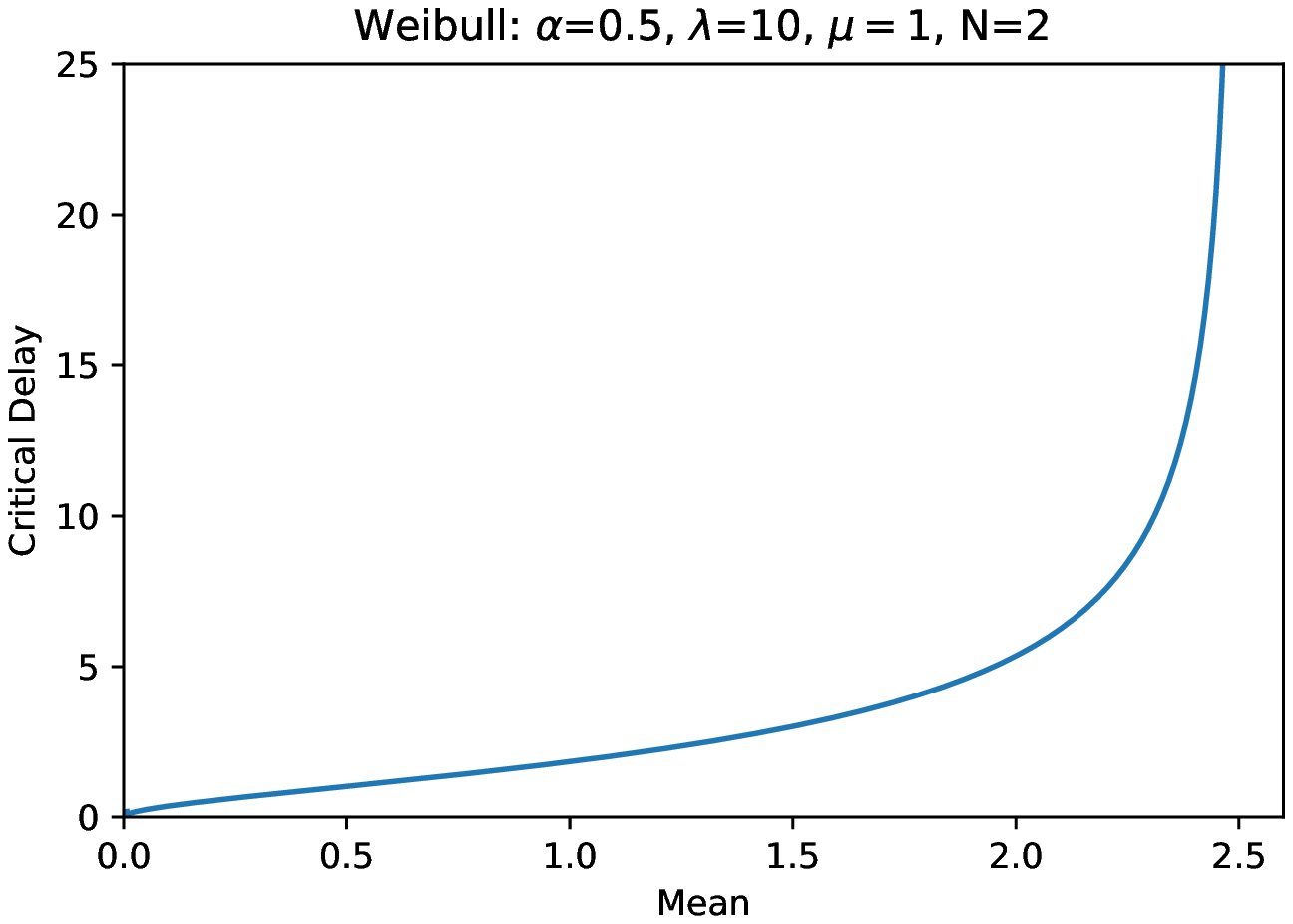} \includegraphics[scale=.6]{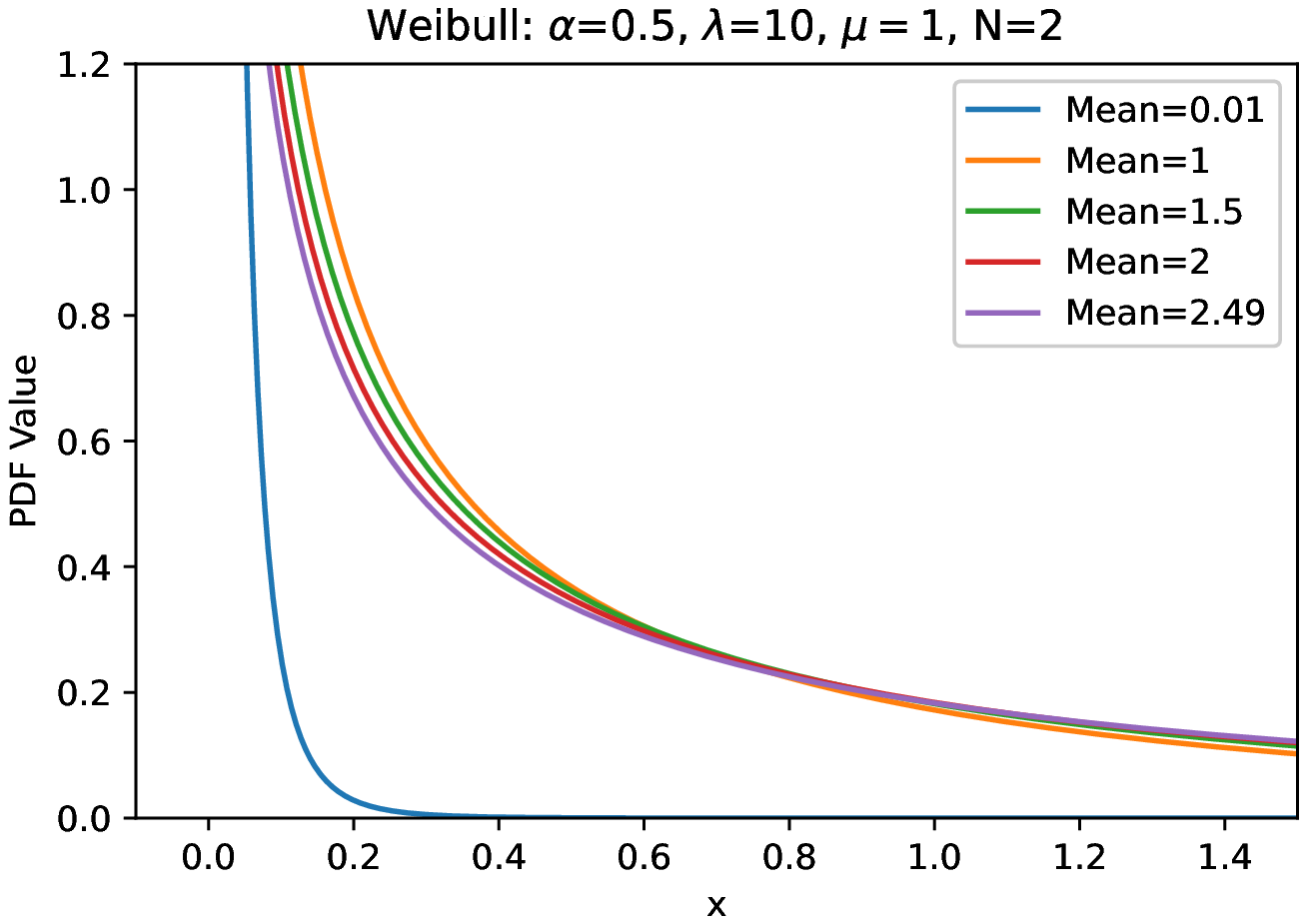}
\caption{Left: The critical delay plotted against the mean of the \textbf{Weibull distribution} that induces the choice model with fixed $\alpha = \frac{1}{2}$. Right: A plot of the probability density function used for some selected values of the mean.}
\label{CHOICE_MODEL_fig_weibull_mean_a_point5}
\end{figure}

\begin{figure}
\hspace{-10mm} \includegraphics[scale=.6]{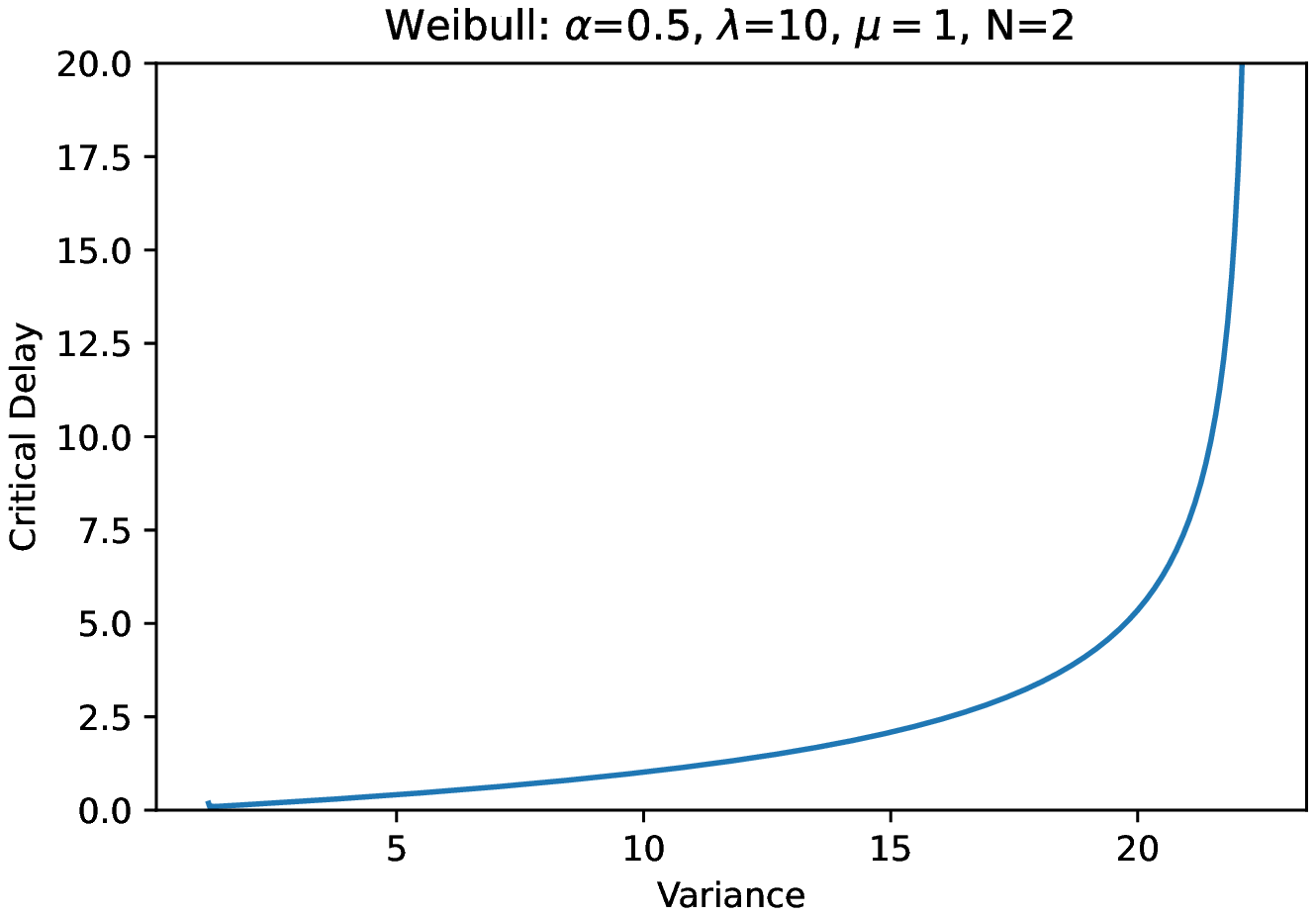} \includegraphics[scale=.6]{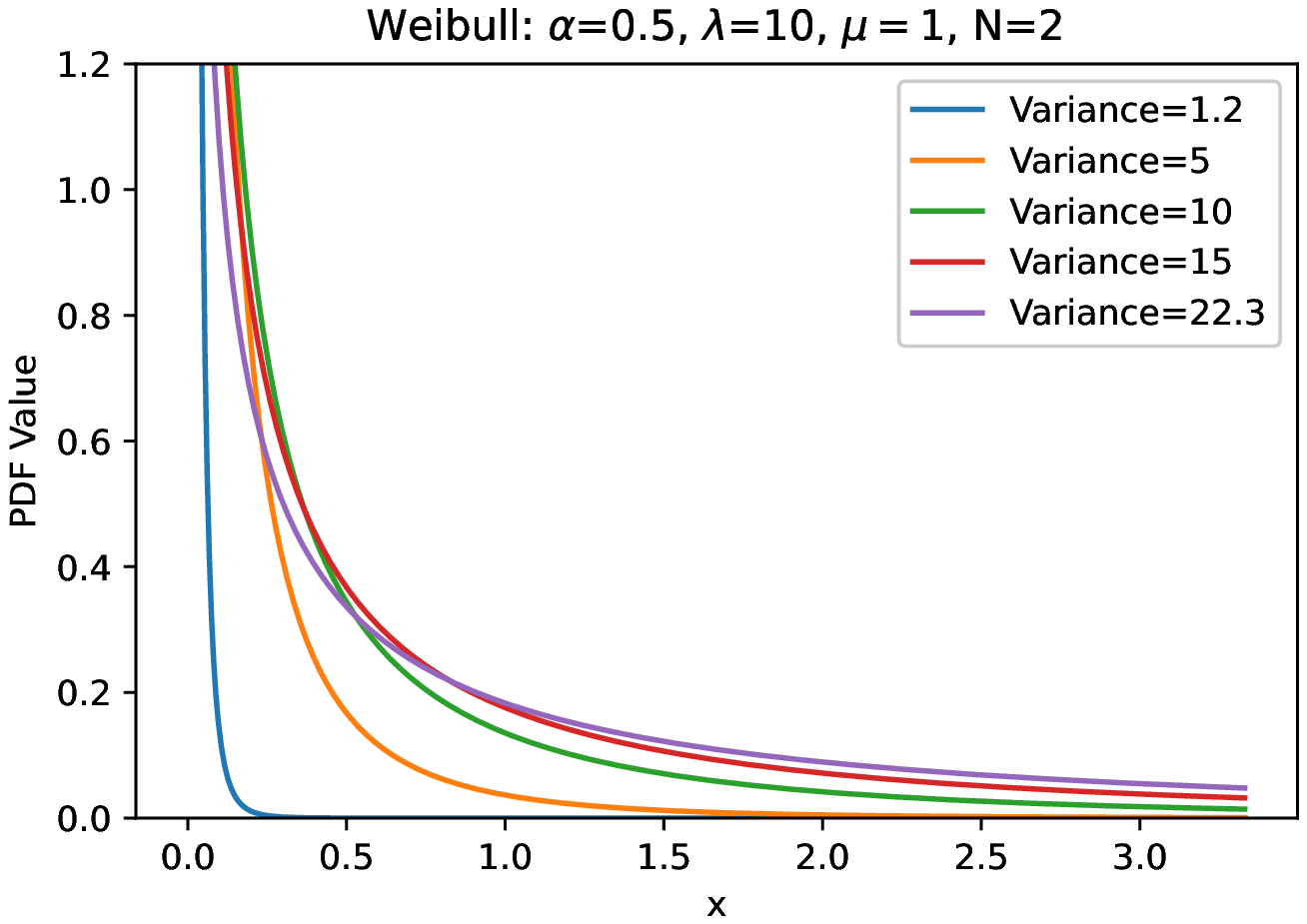}
\caption{Left: The critical delay plotted against the variance of the \textbf{Weibull distribution} that induces the choice model with fixed $\alpha = \frac{1}{2}$. Right: A plot of the probability density function used for some selected values of the variance.}
\label{CHOICE_MODEL_fig_weibull_variance_a_point5}
\end{figure}


\subsection{Gamma Distribution}

In this section, we assume that the complementary cumulative distribution function $\bar{G}$ that characterizes the choice model is given by a gamma distribution. Below we provide some useful quantities relating to the gamma distribution.

\begin{align}
X &\sim \text{Gamma}(\alpha, \beta), \hspace{5mm} \theta > 0\\
g(x) &= \frac{\beta^\alpha}{\Gamma(\alpha)} x^{\alpha - 1} e^{-\beta x }\\
\bar{G}(x) &= 1- \frac{1}{\Gamma(\alpha)} \gamma(\alpha, \beta x)\\
\mathbb{E}[X] &= \frac{\alpha}{\beta}\\
\text{Var}(X) &= \frac{\alpha}{\beta^2}\\
C &= - \frac{\lambda \frac{\beta^\alpha}{\Gamma(\alpha)} \left( \frac{\lambda}{N \mu} \right)^{\alpha - 1} e^{-\beta \frac{\lambda}{N \mu} }}{N \left( 1- \frac{1}{\Gamma(\alpha)} \gamma(\alpha, \beta \frac{\lambda}{N \mu}) \right) }
\end{align}

We consider the gamma distribution which can be viewed as a generalization of the exponential distribution in some sense. In particular, if the parameter $\alpha$ is a positive integer, then the gamma distribution reduces to an Erlang distribution which reduces to an exponential distribution if $\alpha = 1$. Some of the above quantities are defined in terms of the gamma function defined in Equation \ref{CHOICE_MODEL_gamma_function_def}, but we note that some quantities also depend on the lower incomplete gamma function \begin{eqnarray}
\gamma(x, t) := \int_{0}^{t} z^{x - 1} e^{-z} dz, \hspace{5mm} \text{Re}(x) > 0.
\end{eqnarray} Alternatively, one could choose to express quantities involving the lower incomplete gamma function in terms of the upper incomplete gamma function \begin{eqnarray}
\Gamma(x, t) := \int_{t}^{\infty} z^{x - 1} e^{-z} dz, \hspace{5mm} \text{Re}(x) > 0
\end{eqnarray} as we have the apparent relation \begin{eqnarray}
\Gamma(x) = \gamma(x, t) + \Gamma(x, t)
\end{eqnarray} so that one could instead write \begin{eqnarray}
\bar{G}(x) = \frac{1}{\Gamma(\alpha)} \Gamma(\alpha, \beta x) .
\end{eqnarray}
With this in mind, one could alternatively express the eigenvalue as \begin{eqnarray}
C = - \frac{\lambda}{N} \frac{\beta^{\alpha} \left( \frac{\lambda}{N \mu} \right)^{\alpha - 1} e^{ - \beta \frac{\lambda}{N \mu}}}{\Gamma(\alpha, \beta \frac{\lambda}{N \mu})}.
\end{eqnarray}

In Figure \ref{CHOICE_MODEL_fig_gamma} we show queue length plots and phase diagrams on each side of the critical delay. In Figure \ref{CHOICE_MODEL_fig_gamma_mean} we see how the value of the critical delay varies as the mean is varied with a fixed variance and in Figure \ref{CHOICE_MODEL_fig_gamma_variance} we see how it changes as the variance is varied with a fixed mean. For unit variance, we see that there appear to be two critical mean values at which the critical delay becomes unbounded. For unit mean, there appears to be a single critical value of the variance.

\begin{figure}
\begin{tabular}{cc}
  \includegraphics[scale=.55]{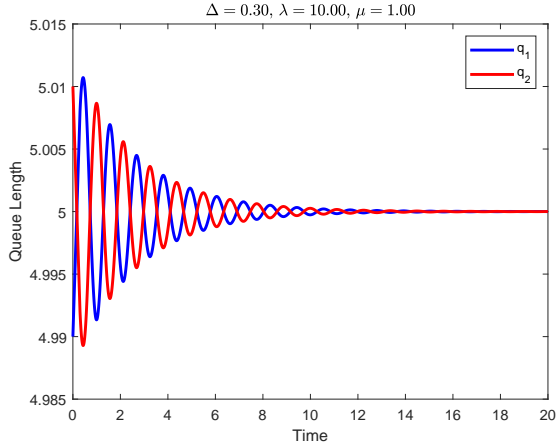} &   \includegraphics[scale=.55]{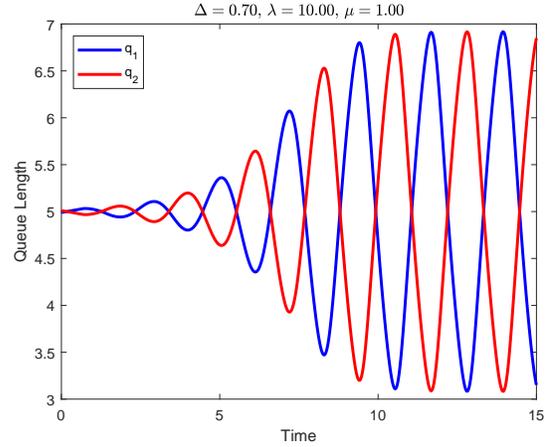} \\
(a)  & (b) \\[6pt]
 \includegraphics[scale=.55]{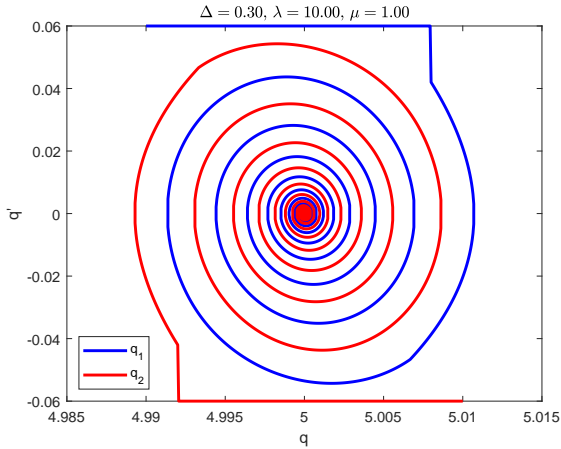} &   \includegraphics[scale=.55]{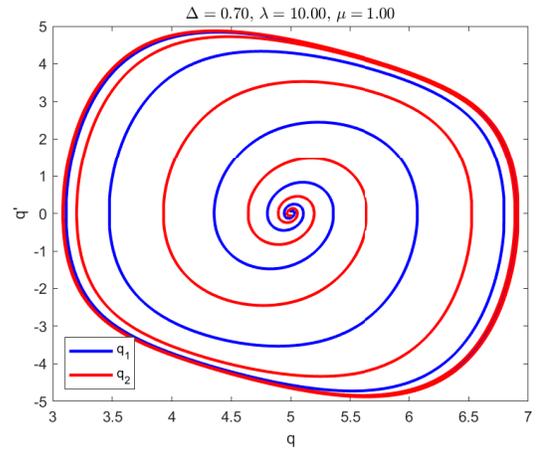} \\
(c)  & (d)  \\[6pt]
\end{tabular}
\caption{Before and after the change in stability using the choice model induced by a \textbf{Gamma distribution} with $\alpha = \beta = 1$ (which reduces to an exponential distribution) with constant history function on $[-\Delta, 0]$ with $q_1 = 4.99$ and $q_2 = 5.01$, $N = 2, \lambda = 10$, $\mu = 1$. The top two plots are queue length versus time with $\Delta = .3$ (a) and $\Delta = .7$ (b). The bottom two plots are phase plots of the queue length derivative with respect to time against queue length for $\Delta = .3$ (c) and $\Delta = .7$ (d). The critical delay is $\Delta_{\text{cr}} = .3617$.}
\label{CHOICE_MODEL_fig_gamma}
\end{figure}

\begin{figure}
\hspace{-10mm} \includegraphics[scale=.6]{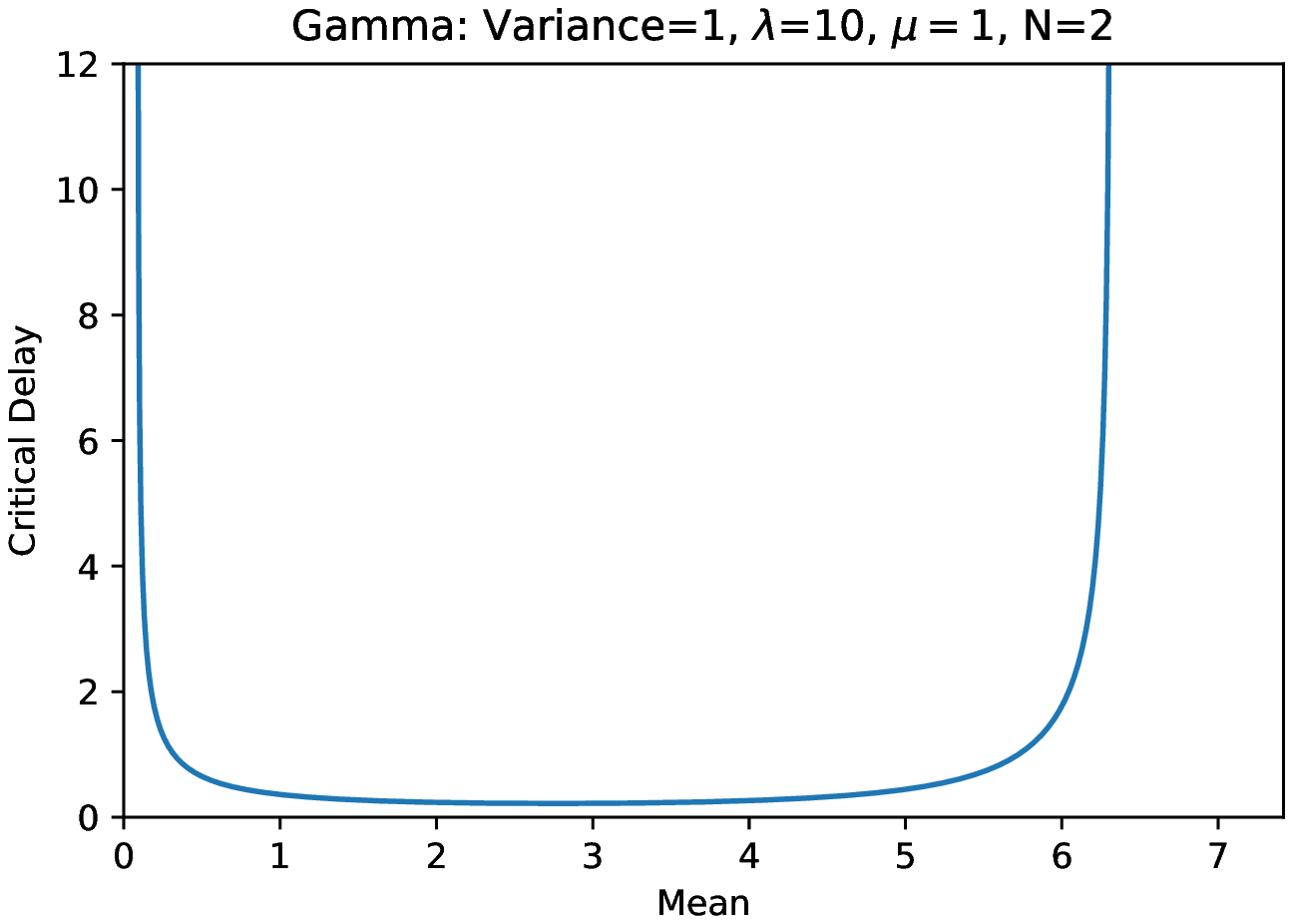} \includegraphics[scale=.6]{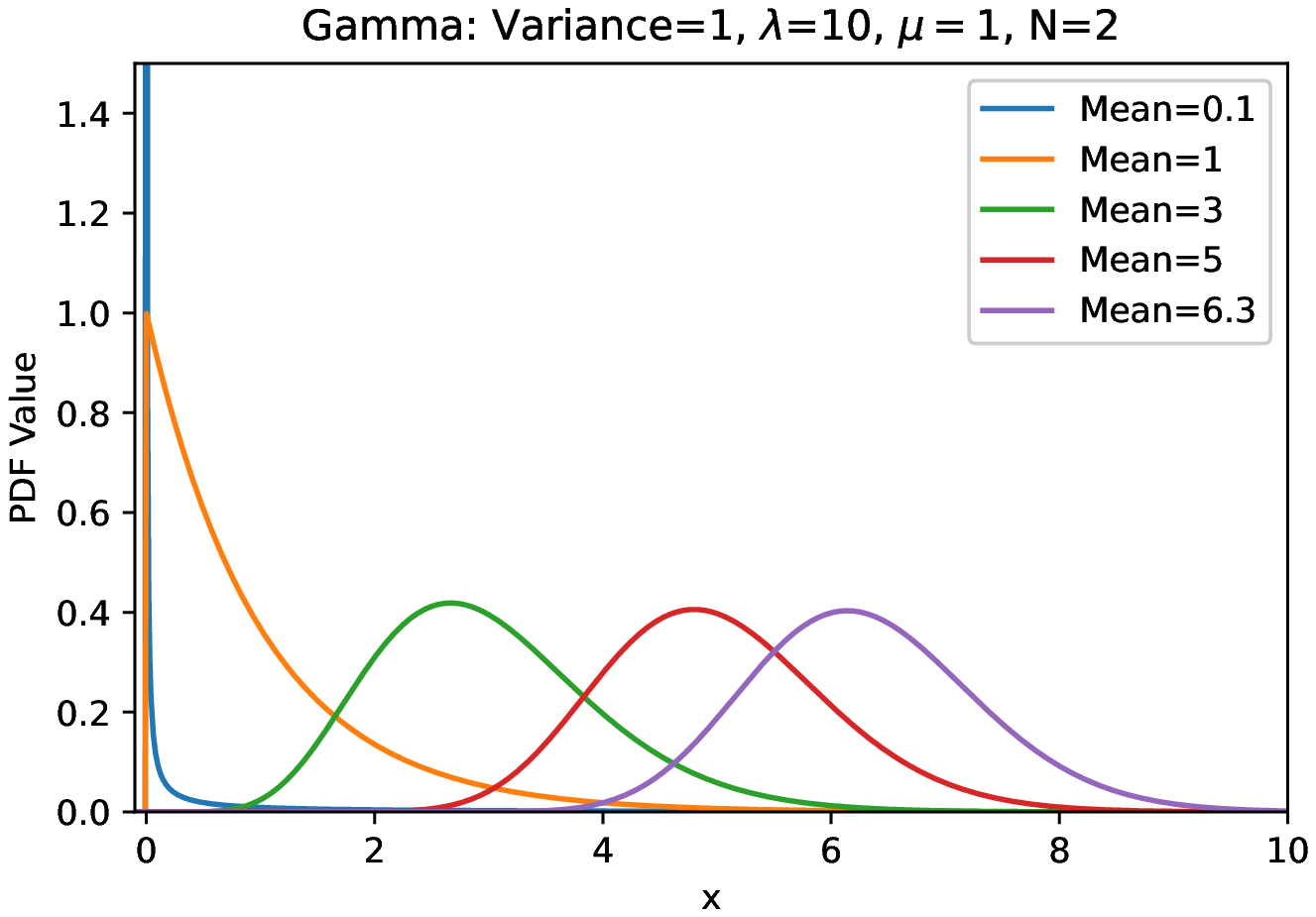}
\caption{Left: The critical delay plotted against the mean of the \textbf{gamma distribution} that induces the choice model with a fixed variance of 1. Right: A plot of the probability density function used for some selected values of the mean.}
\label{CHOICE_MODEL_fig_gamma_mean}
\end{figure}

\begin{figure}
\hspace{-10mm} \includegraphics[scale=.6]{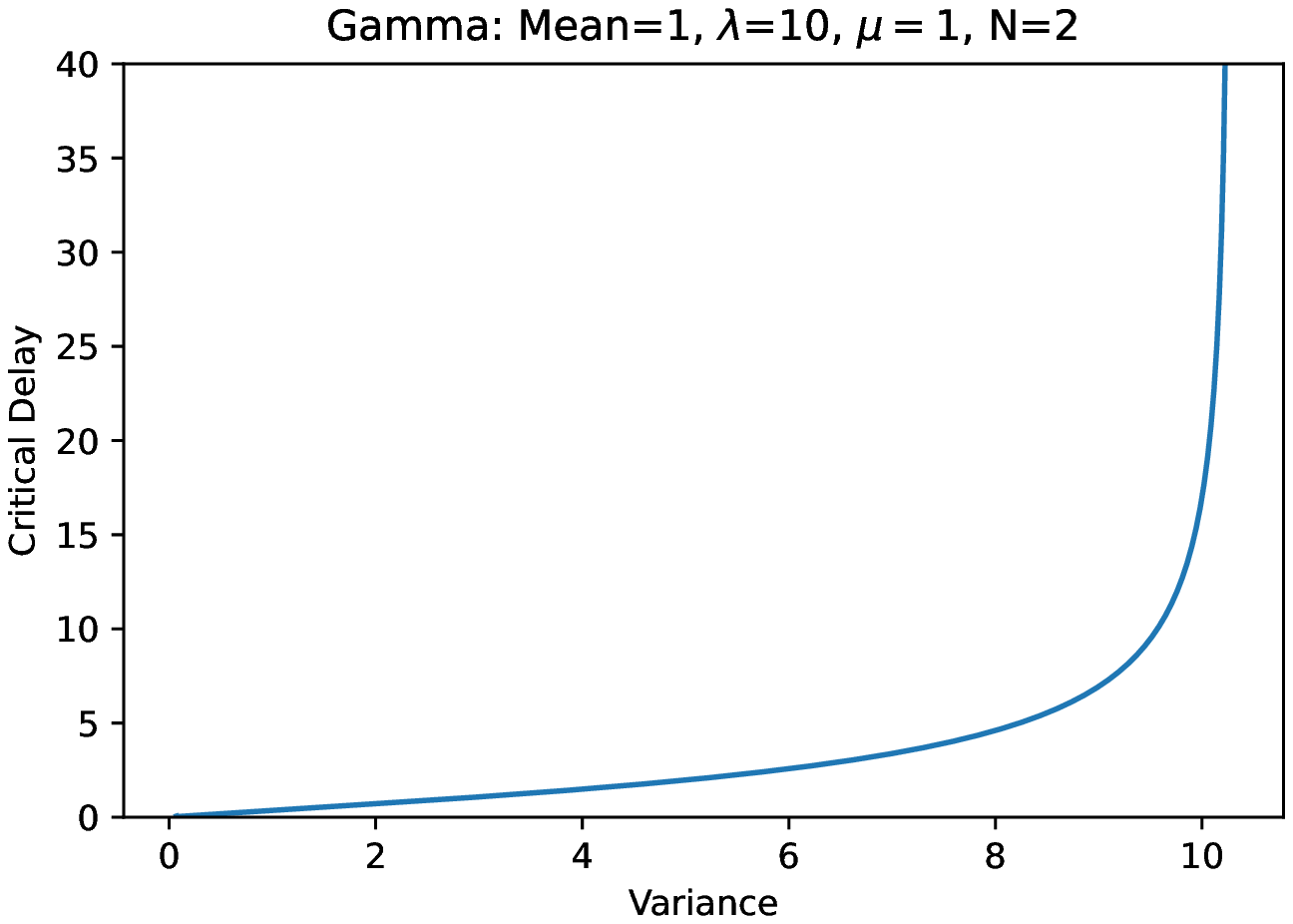} \includegraphics[scale=.6]{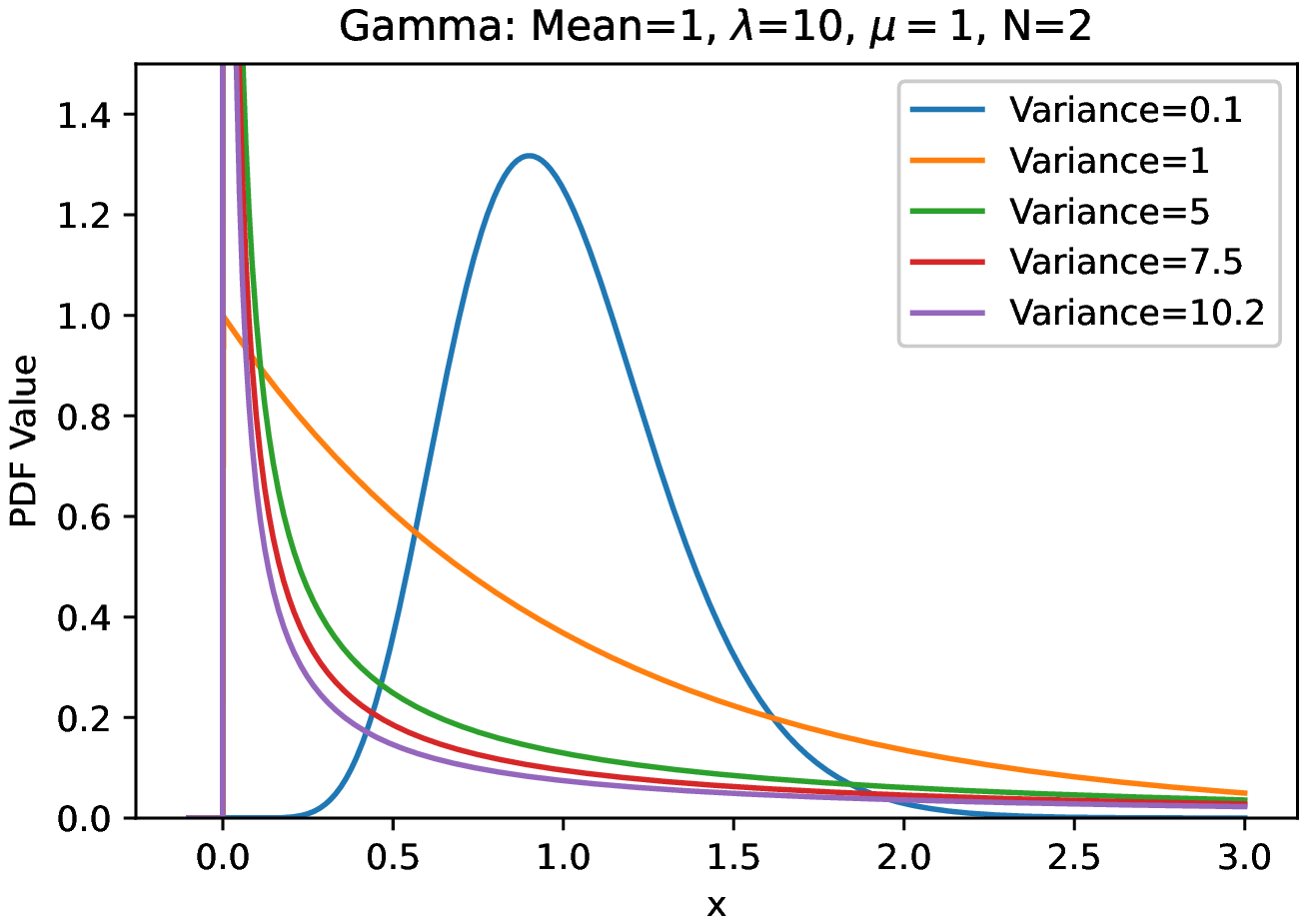}
\caption{Left: The critical delay plotted against the variance of the \textbf{gamma distribution} that induces the choice model with a fixed mean of 1. Right: A plot of the probability density function used for some selected values of the variance.}
\label{CHOICE_MODEL_fig_gamma_variance}
\end{figure}


\subsection{Phase Type Distribution}

In this section, we assume that the complementary cumulative distribution function $\bar{G}$ that characterizes the choice model is given by a phase-type distribution. Below we provide some useful quantities relating to phase-type distributions.

\begin{align}
X &\sim \text{Ph}(\boldsymbol{\alpha}, \mathbf{S^{0}}, S)\\
g(x) &= \boldsymbol{\alpha}\exp({S}x)\mathbf{S^{0}} \\
\bar{G}(x) &= \boldsymbol{\alpha}\exp({S}x)\mathbf{1} \\
\mathbb{E}[X^n] &= (-1)^{n}n!\boldsymbol{\alpha}{S}^{-n}\mathbf{1}\\
\text{Var}(X) &= 2\boldsymbol{\alpha}{S}^{-2}\mathbf{1}-(\boldsymbol{\alpha}{S}^{-1}\mathbf{1})^{2}\\
C &= - \frac{ \lambda \boldsymbol{\alpha}\exp \left({S} \frac{\lambda}{N \mu} \right)\mathbf{S^{0}} }{N \boldsymbol{\alpha}\exp \left({S}\frac{\lambda}{N \mu} \right)\mathbf{1}}
\end{align}

Phase-type distributions describe the distribution of time spent to reach an absorbing state in a continuous-time Markov chain with a finite state space with one absorbing state and the remaining states being transient. Consider a continuous-time Markov chain with $p + 1$ states, $p \in \mathbb{Z}^+$ of which are transient states and the remaining state is an absorbing state. Let the state space be $\{0, 1, ..., p \}$ with state $0$ the absorbing state and let $\boldsymbol{\alpha} \in \mathbb{R}^{1 \times p}$ be a vector of probabilities where the $i^{\text{th}}$ entry of $\boldsymbol{\alpha}$ corresponds to the probability of the continuous-time Markov chain starting at the $i^{\text{th}}$ state. The continuous-time Markov chain has transition-rate matrix \begin{eqnarray}
Q = \begin{bmatrix}
0 & \boldsymbol{0} \\
\boldsymbol{S}^{0} & S
\end{bmatrix}
\end{eqnarray} where $\boldsymbol{0} \in \mathbb{R}^{1 \times p}$ is a vector with each entry $0$, $S \in \mathbb{R}^{p \times p}$, and $\boldsymbol{S}^{0} = - S \boldsymbol{1}$ where $\boldsymbol{1} \in \mathbb{R}^{p \times 1}$ is a vector where each entry is $1$. Special cases of phase-type distributions include the Erlang distribution, which is the distribution of a sum of exponential random variables, and the hyperexponential distribution, which is a mixture of exponential distributions. 

For example, a hyperexponential distribution has probability density function and complementary cumulative distribution function 
\begin{align}
    g(x) &= \sum_{k=1}^{m} p_k \theta_k e^{- \theta_k x}\\
    \bar{G}(x) &= \sum_{k=1}^{m} p_k e^{-\theta_k x}
\end{align}
with $\sum_{k=1}^{m} p_k = 1$ and $p_k \in [0, 1]$ for $k=1,...,m$. It follows that, under the corresponding choice model, the probability of joining the $i^{\text{th}}$ queue is
\begin{eqnarray}
\frac{\bar{G}(q_i(t - \Delta))}{ \sum_{j=1}^{N} \bar{G}(q_j(t - \Delta)) } = \frac{\sum_{k=1}^{m} p_k  e^{-\theta_k q_i(t - \Delta)} }{ \sum_{j=1}^{N} \sum_{k=1}^{m} p_k  e^{-\theta_k q_j(t - \Delta) }}
\end{eqnarray}
for $i=1, ..., N$. It turns out that we can view this choice of $\bar{G}$ as the moment-generating function of a discrete random variable $Y < 0$ that takes on the value $Y = \theta_k$ with probability $p_k$ for $k=1, ..., m$.

In Figure \ref{CHOICE_MODEL_fig_hyperexponential} we show queue length plots and phase diagrams on each side of the critical delay when considering the special case of a hyperexponential distribution. In Figure \ref{CHOICE_MODEL_fig_hyperexponential_mean} we see how the value of the critical delay varies as the mean is varied with a fixed variance and in Figure \ref{CHOICE_MODEL_fig_hyperexponential_variance} we see how it changes as the variance is varied with a fixed mean, both for the case of a hyperexponential distribution. For the cases considered, there appears to be a critical value of the variance as well as a few critical values for the mean.

\begin{figure}
\begin{tabular}{cc}
  \includegraphics[scale=.55]{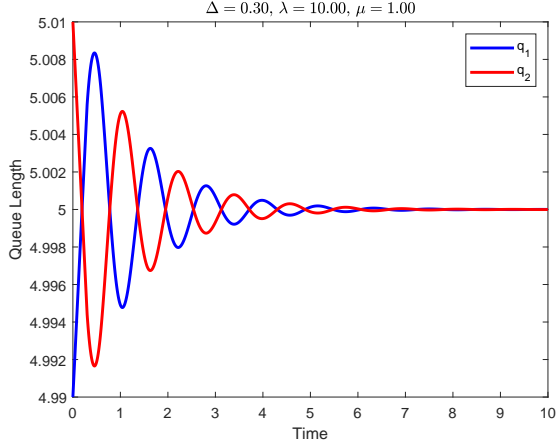} &   \includegraphics[scale=.55]{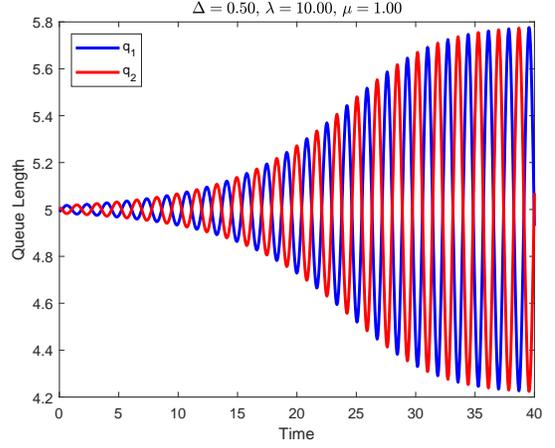} \\
(a)  & (b) \\[6pt]
 \includegraphics[scale=.55]{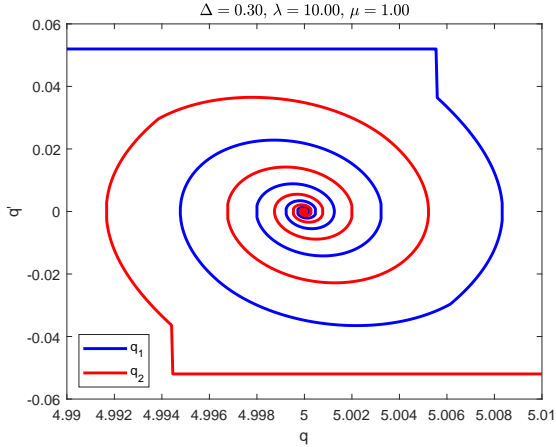} &   \includegraphics[scale=.55]{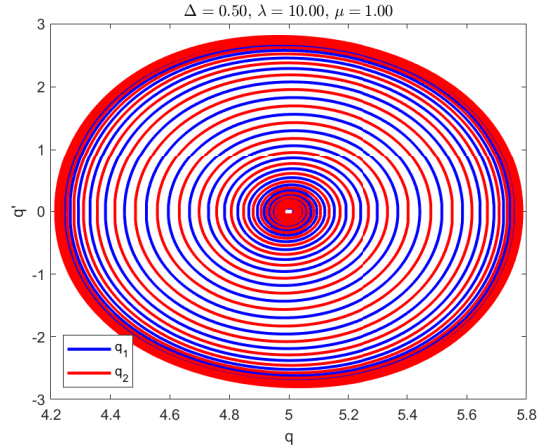} \\
(c)  & (d)  \\[6pt]
\end{tabular}
\caption{Before and after the change in stability using the choice model induced by a \textbf{phase-type distribution} with $\boldsymbol{\alpha} = (.3, .7)$ and $S = \text{diag}(-1.8367, -.8367)$ (which results in a hyperexponential distribution with mean 1 and variance 1) with constant history function on $[-\Delta, 0]$ with $q_1 = 4.99$ and $q_2 = 5.01$, $N = 2, \lambda = 10$, $\mu = 1$. The top two plots are queue length versus time with $\Delta = .3$ (a) and $\Delta = .5$ (b). The bottom two plots are phase plots of the queue length derivative with respect to time against queue length for $\Delta = .3$ (c) and $\Delta = .5$ (d). The critical delay is $\Delta_{\text{cr}} = .4443$.}
\label{CHOICE_MODEL_fig_hyperexponential}
\end{figure}



\begin{figure}
\hspace{-10mm} \includegraphics[scale=.6]{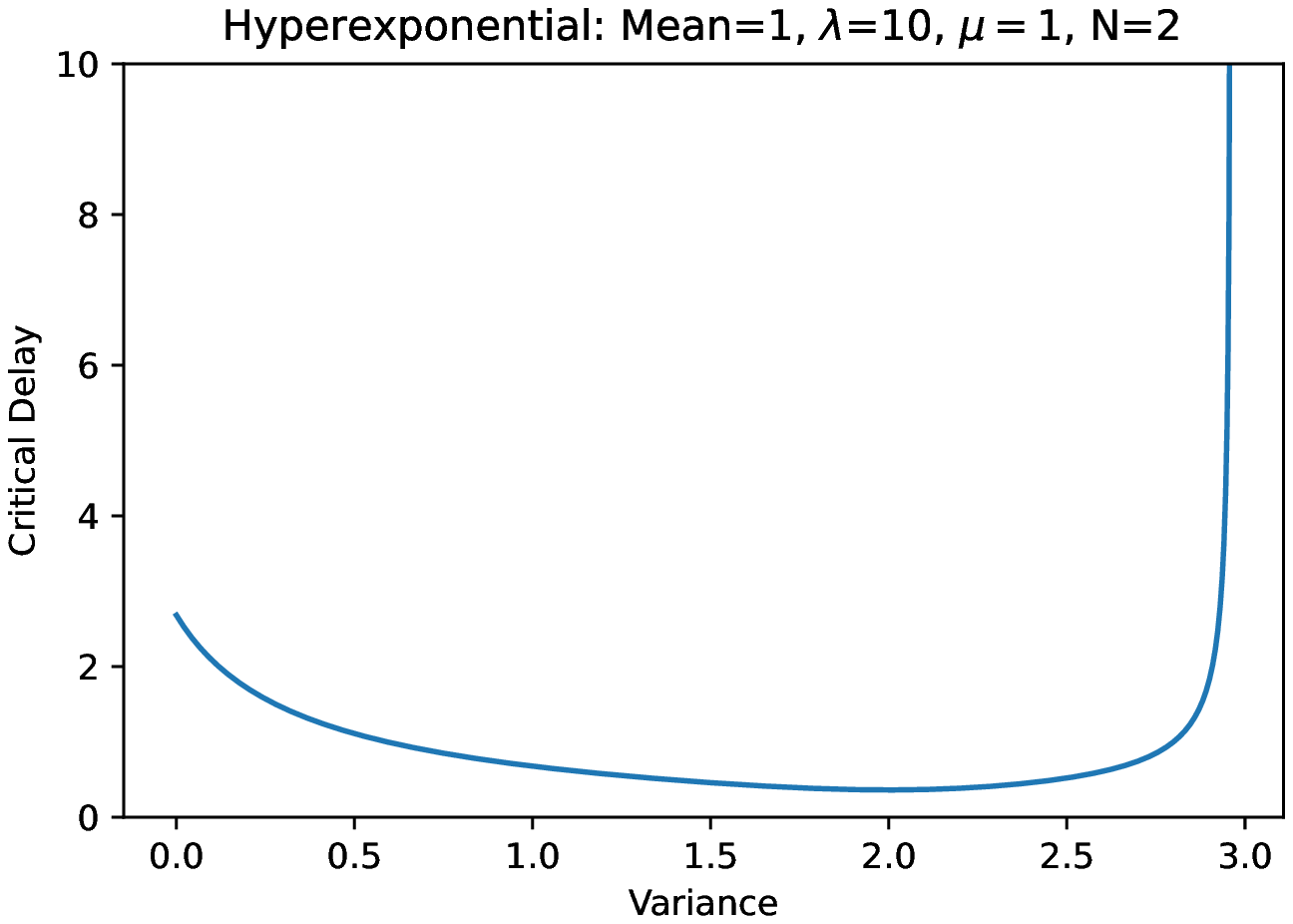} \includegraphics[scale=.6]{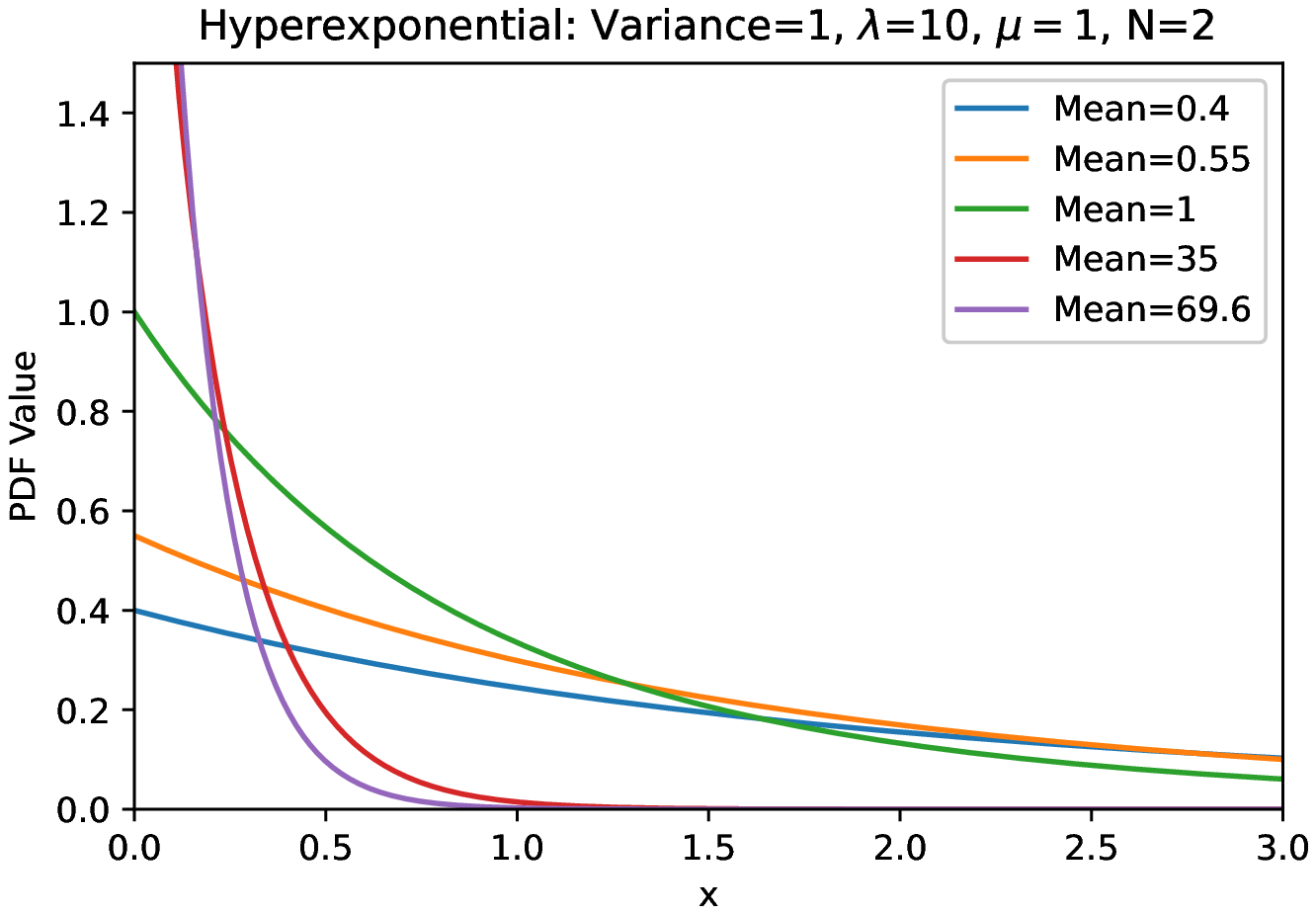}
\caption{Left: The critical delay plotted against the mean of a \textbf{hyperexponential distribution} (in this case, composed of two evenly-weighted exponential distributions) that induces the choice model with a fixed variance of 1. Right: A plot of the probability density function used for some selected values of the mean.}
\label{CHOICE_MODEL_fig_hyperexponential_mean}
\end{figure}

\begin{figure}
\hspace{-10mm} \includegraphics[scale=.6]{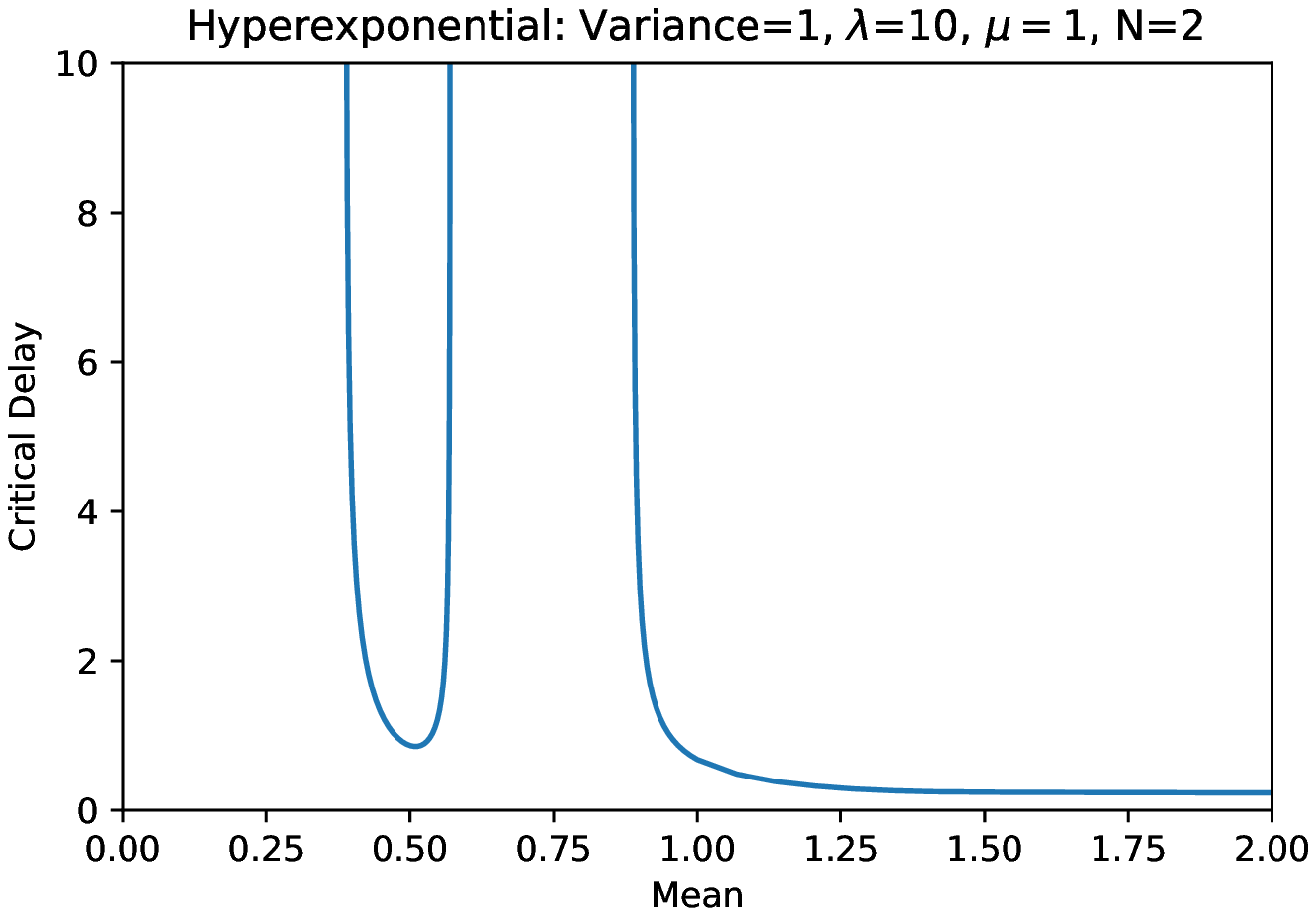} \includegraphics[scale=.6]{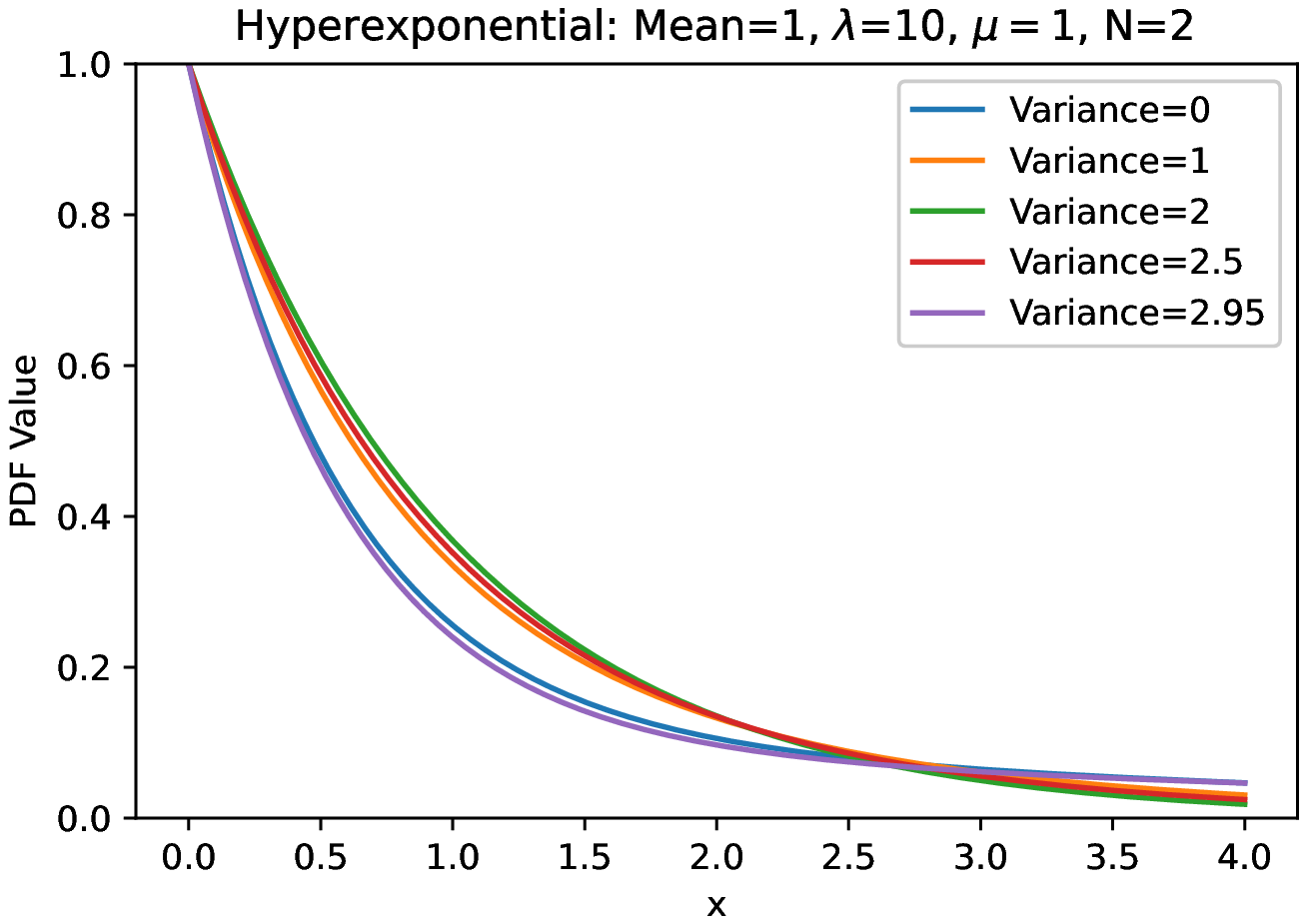}
\caption{Left: The critical delay plotted against the variance of the \textbf{hyperexponential distribution} (in this case, composed of two evenly-weighted exponential distributions) that induces the choice model with a fixed mean of 1. Right: A plot of the probability density function used for some selected values of the variance.}
\label{CHOICE_MODEL_fig_hyperexponential_variance}
\end{figure}

\section{Conclusion} 
\label{CHOICE_MODEL_conclusion_section}
In this paper, we examine choice models informed by utilities that are functions of complementary cumulative distribution functions for some probability distribution and consider an infinite-server fluid queueing system where customers are informed by delayed queue length information. We determine the critical delay of this queueing system in terms of the hazard function of the given probability distribution. We consider how using choice models with functional forms based on various different probability distributions can impact the dynamics of the queueing system. In particular, we see that it is often possible to choose a probability distribution with specific mean and variance to make the critical delay arbitrarily large. This information is useful because such probability distributions can result in a queueing system that is robust to large delays in information. Naturally, there is room for extending this work by considering other information than standard delayed queue length information, such as delayed queue velocity information or updating queue length information. Additionally, other probability distributions could be considered and for larger ranges of parameters. One could also consider other classes of decreasing functions to base the functional form of the choice model on. It could be interesting to focus more attention on choice models induced by an exponential distribution where the distribution parameter has uncertainty and could potentially be viewed as a random variable according to various distributions.

\section{Acknowledgements}

We would like to thank the Center for Applied Mathematics at Cornell University for sponsoring Philip Doldo’s research. Finally, we acknowledge the gracious support of the National Science Foundation (NSF) for Jamol Pender's Career Award CMMI \# 1751975.

\bibliographystyle{plainnat}
\bibliography{main}

\end{document}